\documentclass[12pt]{amsart}
\usepackage[left=25mm,right=25mm,top=25mm,bottom=25mm]{geometry}
\usepackage{palatino}
\usepackage{amsfonts}
\usepackage{amssymb}
\usepackage{amscd}
\usepackage{hyperref}
\usepackage{graphicx}
\usepackage{tikz}
\usepackage{amsmath,amsfonts,amssymb}
\usetikzlibrary{arrows}
\usepackage{color}
\usepackage[T1]{fontenc}

\hypersetup{
    colorlinks=true,
    linkcolor=mediumelectricblue,
    citecolor=carnelian,
    urlcolor=blue,
    linktoc=all
}

\definecolor{ashgrey}{rgb}{0.7, 0.75, 0.71}
\definecolor{oxfordblue}{rgb}{0.0, 0.13, 0.28}
\definecolor{armygreen}{rgb}{0.29, 0.33, 0.13}
\definecolor{bulgarianrose}{rgb}{0.28, 0.02, 0.03}
\definecolor{carnelian}{rgb}{0.7, 0.11, 0.11}
\definecolor{lapislazuli}{rgb}{0.15, 0.38, 0.61}
\definecolor{mediumelectricblue}{rgb}{0.01, 0.31, 0.59}

\newtheorem{thm}{Theorem}[section]
\newtheorem{cor}[thm]{Corollary}

\newtheorem{prop}[thm]{Proposition}
\theoremstyle{definition}

\theoremstyle{remark}
\newtheorem{rem}[thm]{Remark}
\numberwithin{equation}{subsection}

\newtheorem{spec}[thm]{Speculation}
\newtheorem{conj}[thm]{Conjecture}

\newcommand{\al}{\alpha}
\newcommand{\la}{\lambda}
\newcommand{\ka}{\kappa}

\newcommand{\w}{\omega}
\newcommand{\CC}{\mathcal{C}}
\newcommand{\M}{\overline{M}}

\newcommand{\DD}{\mathcal{D}}

\newcommand{\OO}{\mathcal{O}}
\newcommand{\CF}{\mathcal{F}}

\newcommand{\I}{\mathcal{I}}

\newcommand{\E}{\mathbb{E}}
\newcommand{\F}{\mathbb{F}}

\newcommand{\J}{\mathcal{J}}

\newcommand{\QQ}{\mathbb{Q}}

\newcommand{\V}{\mathbb{V}}

\newcommand{\D}{\qquad}

\newcommand{\too}{\rightarrow}

\newcommand{\CM}{\mathcal{M}}
\newcommand{\CR}{\mathcal{R}}

\begin{document}

\title[The moduli space of curves and its invariants]
{The moduli space of curves and its invariants}
\author[M. Tavakol]{Mehdi Tavakol}
   \address{IBS Center for Geometry and Physics}
   \email{mehdi@ibs.re.kr}

\maketitle

%\begin{center}
%\emph{Dedicated to the memory of my mother}
%\end{center}

\begin{abstract}
This note is about invariants of moduli spaces of curves.
It includes their intersection theory and cohomology.
Our main focus is on the distinguished piece containing the so called \emph{tautological classes}.
These are the most natural classes on the moduli space.
We give a review of known results and discuss their conjectural descriptions.
%Their systematic analysis was initiated by Mumford from algebraic geometric point of view.
%The conjecture of Witten proved by Kontsevich lead to a new insight into this subject.
%Since then tautological classes have been studied by many people.
\end{abstract}   

\maketitle

\setcounter{tocdepth}{2}
%\tableofcontents

%\newpage
\part*{Introduction}

The moduli space of a certain class of geometric objects parameterizes the isomorphism classes of these objects. 
Moduli spaces occur naturally in classification problems.  
They appear in many branches of mathematics and in particular in algebraic geometry. 
In this exposition we consider several moduli spaces which involve algebraic curves.
The most basic space in this note is the space $M_g$ which classifies smooth and proper curves of genus $g$.
The moduli space $M_g$ is not a compact space as smooth curves can degenerate into singular ones. 
Deligne and Mumford in \cite{irreducibility} introduced a compactification of this space by means of stable curves.
Stable curves may be singular but only nodal singularities are allowed. 
Such a curve is stable when it is reduced, connected and each smooth rational component meets the other components of the curve in at least 3 points.
This guarantees the finiteness of the group of automorphisms of the curve.  
The moduli space of stable curves of genus $g$ is denoted by $\overline{M}_g$.
It is phrased in \cite{irreducibility} that the moduli space $\overline{M}_g$ is just the $``$underlying coarse variety$"$
of a more fundamental object, the \emph{moduli stack} $\overline{\mathcal{M}}_g$.
For details about algebraic stacks we refer the reader to the references \cite{edidin_stack, everybody, gomez_algebraic}.
An introduction to algebraic stacks is also given in the appendix of \cite{vistoli_intersection}.
A more comprehensive reference is the note \cite{vistoli_notes} by Vistoli.
We only mention that basic properties of algebraic varieties (schemes) and morphisms between them
generalize for algebraic stacks in a natural way. 
It is proven in \cite{irreducibility} that $\overline{\mathcal{M}}_g$ is a separated, proper and smooth algebraic stack of finite type over Spec$(\mathbb{Z})$ and the complement 
$\overline{\mathcal{M}}_g \backslash \mathcal{M}_g$ is a divisor with normal crossings relative to 
Spec$(\mathbb{Z})$. 
They prove that for an algebraically closed field $k$ the moduli space $\overline{M}_g$ over $k$ is irreducible.
%For a reader who is not familiar with this concept we quote a sentence from Witten:
%"What is a stack? It is a space that is locally a quotient".
%This definition fails to describe all kinds of stacks which you can see in the zoo of mathematics.
%However it is the shortest definition which does not suggest a completely wrong picture. 
A bigger class of moduli spaces of curves deals with pointed curves. 
Let $g,n$ be integers satisfying $2g-2+n >0$. 
The moduli space $M_{g,n}$ classifies the isomorphism classes of the objects of the form $(C; x_1, \dots ,x_n)$,
where $C$ is a smooth curve of genus $g$ and the $x_i$'s are $n$ distinct points on the curve. 
The compactification of this space by means of stable $n$-pointed curves is studied by Knudsen and Mumford in a series of papers \cite{knudsen_projectivity_I, knudsen_mumford, knudsen_projectivity_III}.
The main ingredient in showing the projectivity of the moduli space is geometric invariant theory \cite{git}.
For other treatments based on different techniques see articles by Cornalba \cite{cornalba_projectivity}, Geiseker \cite{geiseker_moduli}, Koll\'ar \cite{kollar_moduli} and
Viehweg \cite{viehweg_I, viehweg_II, viehweg_III}.
The notion of an stable $n$-pointed curve is defined as above with slight modifications: 
First of all we assume that all the markings on the curve are distinct smooth points.
For a nodal curve of arithmetic genus $g$ with $n$ marked points the nodes and markings are called 
\emph{special points}. Such a pointed curve is said to be stable if a smooth rational component has at least 3 special points. 
The Deligne-Mumford compactification $\overline{M}_{g,n}$ of $M_{g,n}$ 
parameterizes stable $n$-pointed curves of arithmetic genus $g$.
The space $M_{g,n}$ sits inside $\overline{M}_{g,n}$
as an open dense subset and the boundary $\overline{M}_{g,n} \backslash M_{g,n}$ is a divisor with normal crossings.
There are partial compactifications of $M_{g,n}$ inside $\overline{M}_{g,n}$ which we recall: 
The moduli space of stable $n$-pointed curves of compact type, denoted by $M_{g,n}^{ct}$, 
parameterizes stable curves whose Jacobian is an abelian variety. 
The moduli space $M_{g,n}^{rt}$ of stable $n$-pointed curves with rational tails is the inverse image of $M_g$ under the natural morphism $\overline{M}_{g,n} \rightarrow \overline{M}_g$, when $g \geq 2$. 
The space $M_{1,n}^{rt}$ is defined to be $M_{1,n}^{ct}$ and $M_{0,n}^{rt}=\overline{M}_{0,n}$.

Here we present a short review of well-known facts and conjectures 
about the moduli spaces of curves and their invariants. 
Our main focus is on the study of the tautological classes on these moduli spaces and their images 
in the homology and cohomology groups of the spaces.
%Unfortunately, we don't mention many other interesting aspects of moduli spaces of curves.
%This includes their connection with arithmetic geometry, their birational geometry and their topology.

\vspace{+10pt}
\noindent{\bf Acknowledgments.}
The first version of this note appeared in the introduction of my thesis.
It was revised after a series of lectures at IBS center for geometry and physics.
%I would like to thank Gabriel Drummond-Cole for organising the seminars. 
I would like to thank Carel Faber for the careful reading of the previous version of this note and useful comments.
Essential part of the new topics included in the exposition is based on discussions with many people.
In particular, I would like to thank Petya Dunin-Barkowski, Jérémy Guéré, Felix Janda, Dan Petersen, Aaron Pixton, Alexander Popolitov, Sergey Shadrin, Qizheng Yin and Dmitry Zvonkine.
I was supported by the research grant IBS-R003-S1.

\part*{Intersection theory of moduli spaces of curves}

%\section*{Algebraic stacks and their intersection theory}

\section*{Enumerative Geometry of the moduli space of curves} 
In \cite{mumford_towards} Mumford started the program of studying the enumerative geometry of the set of all curves of arbitrary genus $g$. 
He suggests to take as a model for this the enumerative geometry of the Grassmannians. 
In this case there is a universal bundle on the variety whose Chern classes generate both 
the cohomology ring and the Chow ring of the space. 
Moreover, there are tautological relations between these classes which give a 
complete set of relations for the cohomology and Chow rings.
The first technical difficulty is to define an intersection product in the Chow group of 
$\overline{M}_g$. The variety $\overline{M}_g$ is singular due to curves with automorphisms. 
Mumford solves this problem by observing that $\overline{M}_g$ has the structure of a $Q$-variety. 
A quasi-projective variety $X$ is said to be a $Q$-variety if locally in the \'etale topology it is locally the quotient of a smooth variety by a finite group. 
This is equivalent to saying that there exists an \'etale covering of the quasi-projective variety $X$ by a finite number of varieties each of which is a quotient of a quasi-projective variety by a faithful action of a finite group.
For a $Q$-variety $X$ it is possible to find a global covering $p: \widetilde{X} \rightarrow X$ with a
group $G$ acting faithfully on $\widetilde{X}$ and $X=\widetilde{X}/G$. In this situation there is a covering of $\widetilde{X}$ by a finite number of open subsets $\widetilde{X}_{\alpha}$ and finite groups 
$H_{\alpha}$ such that the collection $X_{\alpha}=\widetilde{X}_{\alpha}/H_{\alpha}$ gives an \'etale covering of $X$. The charts should be compatible in the sense that for all $\alpha$ and $\beta$ the projections
from the normalization $X_{\alpha \beta}$ of the fibered product $X_{\alpha} \times_X X_{\beta}$ to
$X_{\alpha}$ and $X_{\beta}$ should be \'etale. 
An important notion related to a $Q$-variety $X$ is the concept of a $Q$-sheaf. 
By a coherent $Q$-sheaf $\mathcal{F}$ on $X$ we mean a family of coherent sheaves $\mathcal{F}_{\alpha}$ on the local charts $X_{\alpha}$, plus isomorphisms 
$$\mathcal{F}_{\alpha} \otimes_{\mathcal{O}_{X_{\alpha}}} \mathcal{O}_{X_{\alpha \beta}} \cong 
\mathcal{F}_{\beta} \otimes_{\mathcal{O}_{X_{\beta}}} \mathcal{O}_{X_{\alpha \beta}}.$$
Equivalently, a $Q$-sheaf is given by a family of coherent sheaves on the local charts, such that the pull-backs 
to $\widetilde{X}$ glue together to one coherent sheaf $\widetilde{\mathcal{F}}$ on $\widetilde{X}$ on which $G$ acts. 
Mumford proves the following fact about $Q$-sheaves on varieties with a global Cohen-Macaulay cover:

\begin{prop}
If $\widetilde{X}$ is Cohen-Macaulay, then for any coherent sheaf $\mathcal{F}$ on the $Q$-variety $X$,
$\widetilde{\mathcal{F}}$ has a finite projective resolution.
\end{prop}

He shows that $\overline{M}_g$ is globally the quotient of a Cohen-Macaulay variety by a finite group $G$. 
Using the idea of Fulton's Operational Chow ring Mumford defines an intersection product on the Chow group of $\overline{M}_g$ with $\mathbb{Q}$-coefficients. 
More precisely, Mumford proves that for $Q$-varieties with a global Cohen-Macaulay cover there is a canonical isomorphism between the Chow group and the $G$-invariants in the operational Chow ring of the covering variety. 
There are two important conclusions of this result:
First of all there is a ring structure on the Chow group $A^*(X)$, and for all $Q$-sheaves $\mathcal{F}$ on the $Q$-variety $X$, 
we can define Chern classes $$c_k(\mathcal{F}) \in A^k(X).$$

For an irreducible codimension $k$ subvariety $Y$ of an $n$-dimensional $Q$-variety $X=\widetilde{X}/G$ 
as above there are two notions of fundamental class which differ by a rational number.
The usual fundamental class $[Y]$ is an element of the Chow group $A_{n-k}(X)=A^k(X)$.
Another notion is that of the $Q$-class of $Y$, denoted by $[Y]_Q$. This is the class
$ch_k(\mathcal{O}_Y)$ in the $G$-invariant part of the operational Chow ring. 
There is a simple relation between these two classes:
$$[Y]_Q=\frac{1}{e(Y)} \cdot [Y],$$
where the integer $e(Y)$ is defined as follows: 
Consider a collection of local charts $p_{\alpha}: X_{\alpha} \rightarrow X$ such that $X_{\alpha}$
is smooth on which a finite group $G_{\alpha}$ acts faithfully giving an \'etale map $X_{\alpha}/G_{\alpha} \rightarrow X$.
Choose $\alpha$ so that $p_{\alpha}^{-1}$ is not empty. 
Then $e(Y)$ is defined to be the order of the stabilizer of a generic point of 
$p_{\alpha}^{-1}(Y)$ in $G_{\alpha}$.

\begin{rem}
It is indeed true that there are global coverings of the moduli spaces of stable pointed curves by smooth varieties.
This was not known at the time that Mumford defined the intersection product on the Chow groups of 
$M_g$ and $\overline{M}_g$. This covering is defined in terms of level structures on curves.
The notion of a genus $g$ curve with a Teichm\"uller structure of level $G$ has been first introduced by Deligne and
Mumford in \cite{irreducibility}.
Looijenga \cite{looijenga_smooth} introduces the notion of a Prym level structure on a smooth curve and shows that the compactification
$X$ of the moduli space of curves with Prym level structure is smooth and that $\overline{M}_g=X/G$, 
with $G$ a finite group. 
In \cite{boggi_pikaart} Boggi and Pikaart prove that $\overline{M}_{g,n}$ is a quotient of a smooth variety by a finite group.  
Therefore, it is easier to prove the existence of an intersection product with the desired properties on the 
moduli spaces of curves using standard intersection theory. 
\end{rem}

\subsubsection{Tautological classes}
To define the tautological classes Mumford considers the 
coarse moduli space of 1-pointed stable curves of genus 
$g$, denoted by $\overline{M}_{g,1}$ or $\overline{C}_g$. 
$\overline{C}_g$ is also a $Q$-variety and has a covering $\widetilde{C}_g$
and there is a morphism  
$$\pi: \widetilde{C}_g \rightarrow \widetilde{M}_g,$$
which is a flat, proper family of stable curves, with a finite group
$G$ acting on both, and 
$\overline{C}_g=\widetilde{C}_g/G$, 
$\overline{M}_g=\widetilde{M}_g/G$. 
There is a $Q$-sheaf 
$\omega:=\omega_{\overline{C}_g/\overline{M}_g}$
induced from the dualizing sheaf 
$\omega_{\widetilde{C}_g/\widetilde{M}_g}$
of the projection $\pi$. 
The tautological classes are defined by:
$$K=c_1(\omega) \in A^1(\overline{C}_g), $$
$$\kappa_i=\pi_* (K^{i+1}) \in A^i(\overline{M}_g),$$
$$\mathbb{E}= \pi_*(\omega): \qquad \text{a locally free} \  Q\text{-sheaf of rank} \  g \ \text{on} \ \overline{M}_g,$$ 
$$\lambda_i=c_i(\mathbb{E}), \qquad 1 \leq i \leq g.$$

\subsubsection{Tautological relations via Grothendieck-Riemann-Roch}
To find tautological relations among these classes Mumford applies the Grothendieck-Riemann-Roch theorem 
to the morphism $\pi$ and the class $\omega$. 
This gives an identity which expresses the Chern character of the
Hodge bundle $\mathbb{E}$ in terms of the kappa classes and the boundary cycles.  
He also proves that for all even positive integers $2k$, $$(ch \ \mathbb{E})_{2k}=0.$$
As a result one can express the even $\lambda_i$'s as polynomials in terms of the odd ones, 
and all the $\lambda_i$'s as polynomials in terms of $\kappa$ classes and boundary cycles.  
Another approach in finding relations between the $\lambda$ and $\kappa$ 
classes is via the canonical linear system.
In this part Mumford considers smooth curves. The method is based on the fact 
that for all smooth curves $C$, the canonical sheaf $\omega_C$ is generated by its global sections. 
This gives rise to an exact sequence of coherent shaves on the moduli space and leads to the following result:

\begin{cor}\label{lam}
For all $g$, all the classes $\lambda_i,\kappa_i$ restricted to $A^*(M_g)$ are polynomials in $\kappa_1, \dots , \kappa_{g-2}$.
\end{cor}

\subsubsection{The tautological classes via Arbarello's flag of subvarieties of $M_g$}
In this part Mumford considers the following subsets of $C_g$ and $M_g$:
$$W_l^*=\{(C,x) \in C_g: h^0(\mathcal{O}_C( l \cdot x)) \geq 2 \} \subset C_g,$$
$$W_l=\pi(W_l^*) \subset M_g,$$
where $2 \leq l \leq g$. He proves that $W_l^*$ is irreducible of codeminsion $g-l+1$ 
and its class is expressed in terms of tautological classes:
$$[W_l^*]_Q=(g-l+1)^{st}- \text{component of}$$
$$\pi^*(1-\lambda_1+\lambda_2- \dots +(-1)^g \lambda_g) \cdot (1-K)^{-1} \cdot \dots \cdot (1-lK)^{-1}.$$
Considering $W_l$ as the cycle $\pi_*(W_l^*)$, it follows that
$$[W_l]_Q=(g-l)^{st}- \text{component of}$$
$$(1-\lambda_1+\lambda_2- \dots +(-1)^g \lambda_g) \cdot \pi_*[ (1-K)^{-1} \cdot \dots \cdot (1-lK)^{-1}].$$
In the special case $l=2$ the subvariety $W_2$ is the locus of hyperelliptic curves.  
The formula for the class of this locus is as follows:
$$[H]=2[H]_Q=\frac{1}{g+1}\left( (2^g-1)\kappa_{g-2} - \dots +(-1)^{g-2}(6g-6)\lambda_{g-2} \right).$$

In the third part of \cite{mumford_towards} Mumford gives a complete description of the intersection ring of $\overline{M}_2$. 
He proves the following fact:

\begin{thm}
The Chow ring of $\overline{M}_2$ is given by $$A^*(\overline{M}_2)=\frac{\mathbb{Q}[\lambda, \delta_1]}{(\delta_1^2+\lambda \delta_1, 5 \lambda^3-\lambda^2 \delta_1)}.$$
The dimensions of the Chow groups are 1,2,2 and 1, and the pairing between Chow groups of complementary dimensions is perfect.
\end{thm}

\subsection*{Known results}
In this part we review known results about the intersection theory of the moduli spaces of curves after Mumford. 
In the first part of his thesis \cite{faber_thesis} Faber studies the intersection ring of the moduli space 
$\overline{M}_3$ and he proves the following result:
 
\begin{thm} 
The Chow ring of $\overline{M}_3$ is given by 
$$A^*(\overline{M}_3)=\frac{\mathbb{Q}[\lambda, \delta_0, \delta_1, \kappa_2]}{I},$$
where $I$ is generated by three relations in codimension 3 and six relations in codimension 4. 
The dimensions of the Chow groups are 1,3,7,10,7,3,1, respectively. 
The pairing $$A^k(\overline{M}_3) \times A^{6-k}(\overline{M}_3) \rightarrow \mathbb{Q}$$
is perfect.
\end{thm}

Faber gives a complete description of the intersection ring of $\overline{M}_3$. 
This ring is generated by tautological classes.
Faber also determines the ample divisor classes on $\overline{M}_3$. 
In the second part of the thesis Faber obtains partial results with respect to $\overline{M}_4$. 
In particular, he shows that the Chow ring of $M_4$ is $\mathbb{Q}[\lambda]/(\lambda^3)$ and he proves that the dimension of $A^2(\overline{M}_4)$ equals 13. 
Another interesting question would be to fix a genus $g$ and study the intersection ring of the moduli spaces 
$M_{g,n}$ and $\overline{M}_{g,n}$ when $n$ varies. 
The most basic case is of course the classical case of genus zero.  
The first description of the intersection ring of $\overline{M}_{0,n}$ is due to Keel \cite{keel_intersection}. 
Keel shows that the variety $\overline{M}_{0,n}$ can be obtained as a result of a sequence of 
blowups of smooth varieties along smooth codimension two subvarieties. 
In his inductive approach Keel describes $\overline{M}_{0,n+1}$ as a blow-up of 
$\overline{M}_{0,n} \times \mathbb{P}^1$, assuming that $\overline{M}_{0,n}$ is already constructed. 
For example, $\overline{M}_{0,3}$
consists of a single point, $\overline{M}_{0,4}$ is the projective line $\mathbb{P}^1$, 
and $\overline{M}_{0,5}$ is obtained from $\mathbb{P}^1 \times \mathbb{P}^1$ by blowing up 3 points, etc.  
He proves the following facts:

\begin{itemize}
\item The canonical map from the Chow groups to homology (in characteristic zero) 
$$A_*(\overline{M}_{0,n}) \rightarrow H_*(\overline{M}_{0,n})$$ 
is an isomorphism.

\item There is a recursive formula for the Betti numbers of $\overline{M}_{0,n}$.

\item There is an inductive recipe for determining dual bases in the Chow ring $A^*(\overline{M}_{0,n})$.

\item The Chow ring is generated by the boundary divisors. More precisely, 
$$A^*(\overline{M}_{0,n})= \frac{ \mathbb{Z} [D_S: S \subset \{1 , \dots , n\} , \ |S|,|S^c| \geq 2]}{I},$$ 
where $I$ is the ideal generated by the following relations:

\begin{enumerate}
\item[(1)] $D_S=D_{S^c}$,

\item[(2)] For any four distinct elements $i,j,k,l \in \{1,\dots ,n\}$: 
$$\sum_{\substack{i,j \in I\\k,l \notin I}}D_I=\sum_{\substack{i,k \in I\\j,l \notin I}}D_I=\sum_{\substack{i,l \in I\\j,k \notin I}}D_I,$$
 
\item[(3)] $$D_I \cdot D_J=0 \qquad \text{unless} \qquad I \subseteq J, \qquad J \subseteq I,  \qquad 
\text{or} \ I \cap J=\emptyset.$$
\end{enumerate}
\end{itemize}

According to Keel's result all relations between algebraic cycles on $\overline{M}_{0,n}$ follow from the basic fact that any two points of the projective line $\overline{M}_{0,4}=\mathbb{P}^1$ are rationally equivalent. 
While Keel gives a complete description of the intersection ring, the computations and exhibiting a basis for the Chow groups very quickly become involved as $n$ increases. 
In \cite{tavakol_0}, we give another construction of the moduli space $\overline{M}_{0,n}$ as a blowup of the variety $(\mathbb{P}^1)^{n-3}$. This gives another proof of Keel's result. This construction gives an explicit basis for the Chow groups and an explicit duality between the Chow groups in complementary degrees. Another advantage of this approach is that fewer generators and fewer relations are needed to describe the intersection ring. This simplifies the computations. 
In \cite{belorousski_chow} Belorousski considers the case of pointed elliptic curves. 
He proves that $M_{1,n}$ is a rational variety for $n \leq 10$.
One has the isomorphism $A_n(M_{1,n}) \cong \mathbb{Q}$ 
since $M_{1,n}$ is an irreducible variety of dimension $n$. Belorousski proves the following result: 
\begin{thm}\label{m}
$A_*(M_{1,n})= \mathbb{Q}$ for $1 \leq n \leq 10$.
\end{thm} 

After proving this result he considers the compactified moduli spaces $\overline{M}_{1,n}$ for the case $1 \leq n \leq 10$. 
Belorousski proves the following fact about the generators of the Chow groups:

\begin{thm}
For $1 \leq n \leq 10$ the Chow group $A_*(\overline{M}_{1,n})$ is spanned by boundary cycles.
\end{thm}
The method of proving the statement is to consider the standard stratification of the moduli space $\overline{M}_{1,n}$ 
by the topological type of the stable curves. 
In this part he uses the result of Theorem \ref{m} about the Chow groups of the open parts. 
Belorousski also computes the intersection ring of $\overline{M}_{1,n}$ for $n=3,4$. 
In each case the intersection ring coincides with the tautological algebra and the intersection pairings are perfect.
Belorousski obtains partial results on the intersection rings of $\overline{M}_{1,5}$
and $\overline{M}_{1,6}$. He shows that the Chow ring of $\overline{M}_{1,5}$ is generated by boundary divisors. 
In the study of the Chow ring of $\overline{M}_{1,4}$ there is a codimension two relation which plays an important
role. This relation first appeared in the article \cite{getzler_intersection} by Getzler.
In the third chapter of his thesis Faber computes the Chow ring of the moduli space $\overline{M}_{2,1}$.
He proves that the dimensions of the Chow groups $A^k(\overline{M}_{2,1})$ for $k=0, \dots , 4$ are 1,3,5,3,1 respectively.
Faber describes the Chow ring of $\overline{M}_{2,1}$ as an algebra generated by tautological classes 
with explicit generators and relations. He also proves that the pairings 
$$A^k(\overline{M}_{2,1}) \times A^{4-k}(\overline{M}_{2,1}) \rightarrow \mathbb{Q}$$ are perfect for all $k$.

\subsection{Witten's conjecture}
Define $F_g$ by 
$$F_g:=\sum_{n \geq 0} \frac{1}{n!} \sum_{k_1, \dots , k_n} \left( \int_{\overline{M}_{g,n}} \psi_1^{k_1} \dots \psi_n^{k_n} \right) t_{k_1} \dots t_{k_n},$$
the generating function for all top intersections of $\psi$-classes in genus $g$. 
Define a generating function for all such 
intersections in all genera by $$F:=\sum_{g=0}^{\infty} F_g \lambda^{2g-2}.$$
This is Witten's free energy or the Gromov-Witten potential of a point. 
Witten's conjecture (Kontsevich's Theorem) gives a recursion for top intersections of $\psi$-classes in the form of a partial differential equation satisfied by $F$. Using this differential equation along with a geometric fact known as the string equation and the initial condition $\int_{\overline{M}_{0,3}}1=1$, all top intersections are recursively determined.
Witten's conjecture (Kontsevich's Theorem) is:
$$(2n+1)\frac{\partial ^3}{\partial t_n \partial t_0^2}F= \left( \frac{\partial^2}{\partial t_{n-1} \partial t_0} F \right)
\left( \frac{\partial^3}{\partial t_0^3} F \right) + 2\left( \frac{\partial^3}{\partial t_{n-1} \partial t_0^2} F \right)
\left( \frac{\partial^2}{\partial t_0^2} F \right)+ \frac{1}{4} \frac{\partial^5}{\partial t_{n-1} \partial t_0^4}F.$$
This conjecture was stated in \cite{witten_two} by Witten. 
The first proof of Witten's conjecture was given by Kontsevich \cite{kontsevich_intersection}. 
There are many other proofs, e.g., by Okounkov-Pandharipande \cite{okounkov_pandharipande_hurwitz}, Mirzakhani \cite{mirzakhani_weil}, and Kim-Liu \cite{kim_liu_simple}. For more details about Witten's conjecture see 
\cite{vakil_moduli, vakil_gromov_witten}.
From Witten-Kontsevich one can compute all intersection numbers of the 
$n$ cotangent line bundles on the moduli
space $\overline{M}_{g,n}$. In \cite{faber_algorithms} it is proven that the knowledge of these numbers suffices for the computation of other intersection numbers, when arbitrary boundary divisors are included. 
An algorithm for the computation 
of these numbers is given in \cite{faber_algorithms}. 

%\subsection{Virasoro conjecture}

%\subsection{ELSV formula}

\part*{Tautological rings of moduli spaces of curves}
While it is important to understand the structure of the whole intersection ring of the 
moduli spaces of curves, their study seems to be a difficult question at the moment. 
In fact the computation of the Chow rings becomes quite involved even in genus one. 
On the other hand the subalgebra of the Chow ring generated by the tautological classes seems 
to be better behaved and have an interesting conjectural structure. 
The tautological rings are defined to be the subalgebras of the rational Chow rings generated by the tautological classes. 
Faber has computed the tautological rings of the moduli spaces of curves in many cases. 
In \cite{faber_conjectural} he formulates a conjectural description of the tautological ring of $M_g$. 
There are similar conjectures by Faber and Pandharipande about the structure of 
the tautological algebras in the more general setting
when stable curves with markings are included as well. 
According to the Gorenstein conjectures the tautological algebras satisfy a form of Poincar\'e duality. 
In the following, we recall the relevant  definitions and basic facts about the tautological ring of the 
moduli spaces of curves and their conjectural structure. 
We review the special case of $M_g$ more closely. 
Then we discuss the other evidence in low genus when we consider curves with markings. 

\subsection{The tautological ring of $\overline{M}_{g,n}$}
The original definition of the tautological classes on the moduli spaces $M_g$ and $\overline{M}_g$ is due to Mumford.
There is a natural way to define the tautological algebras 
for the more general situation of stable $n$-pointed curves.
In \cite{faber_pandharipande_relative} the system of tautological rings is defined to 
be the set of smallest $\mathbb{Q}$-subalgebras of the Chow rings, $$R^*(\overline{M}_{g,n}) \subset A^*(\overline{M}_{g,n}),$$satisfying the following two properties:

\begin{itemize}
\item
The system is closed under push-forward via all maps forgetting markings: 
$$\pi_*:R^*(\overline{M}_{g,n}) \rightarrow R^*(\overline{M}_{g,n-1}).$$

\item
The system is closed under push-forward via all gluing maps: 
$$\iota_*: R^*(\overline{M}_{g_1,n_1 \cup \{*\}}) \otimes  R^*(\overline{M}_{g_2,n_2 \cup \{ \bullet \}})  \rightarrow R^*(\overline{M}_{g_1+g_2,n_1+n_2}),$$
$$\iota_*: R^*(\overline{M}_{g,n \cup \{ *,\bullet \}}) \rightarrow R^*(\overline{M}_{g+1,n}),$$ 
with attachments along the markings $*$ and $\bullet.$
\end{itemize}

\begin{rem}
It is important to notice that natural algebraic constructions of cycles on the moduli space yield Chow classes in the tautological ring. The first class of such cycles comes from the cotangent line classes. 
For each marking $i$, there is a line bundle $\mathbb{L}_i$ on the moduli space $\overline{M}_{g,n}$. 
The fiber of $\mathbb{L}_i$ at the moduli point $[C;x_1, \dots , x_n]$ is the cotangent space to 
$C$ at the $i^{th}$ marking.
This gives the divisor class $$\psi_i=c_1(\mathbb{L}_i) \in A^1(\overline{M}_{g,n}).$$
The second class comes from the push-forward of powers of $\psi$-classes:
$$\kappa_i=\pi_*(\psi_{n+1}^{i+1}) \in A^i(\overline{M}_{g,n}),$$
where $\pi:\overline{M}_{g,n+1} \rightarrow \overline{M}_{g,n}$ forgets the last marking on the curve. 
Another class of important cycles on $\overline{M}_{g,n}$ is obtained from the Hodge bundle. 
The Hodge bundle $\mathbb{E}$ over $\overline{M}_{g,n}$ is the rank $g$ vector bundle whose fiber 
over the moduli point $[C;x_1, \dots , x_n]$ is the vector space $H^0(C,\omega_C)$, where $\omega_C$ 
is the dualizing sheaf of $C$. 
The Chern classes of the Hodge bundle give the lambda classes:
$$\lambda_i=c_i(\mathbb{E}) \in A^i(\overline{M}_{g,n}).$$
It is a basic fact that all the $\psi,\kappa,\lambda$ classes belong to the tautological ring. 
\end{rem}

\subsection{Additive generators of the tautological rings}
It is possible to give a set of additive generators for the tautological algebras indexed by the boundary cycles. 
This also shows that the tautological groups have finite dimensions.
The boundary strata of the moduli spaces of curves correspond
to {\em stable graphs} $$A=(V, H,L, g:V \rightarrow {\mathbb Z}_{\geq 0}, a:H\rightarrow V, i: H\rightarrow H)$$
satisfying the following properties:
\begin{enumerate}
\item[(1)] $V$ is a vertex set with a genus function $g$,
\item[(2)] $H$ is a half-edge set equipped with a
vertex assignment $a$ and fixed point free involution
$i$,
\item[(3)] $E$, the edge set, is defined by the
orbits of $i$ in $H$, 
\item[(4)] $(V,E)$ define a {\em connected} graph,
\item[(5)] $L$ is a set of numbered legs attached to the vertices,
\item[(6)] For each vertex $v$, the stability condition holds:
$$2g(v)-2+ n(v) >0,$$
where $n(v)$ is the valence of $A$ at $v$ including both half-edges and legs.
\end{enumerate}

Let $A$ be a stable graph. The genus of $A$ is defined by
$$g= \sum_{v\in V} g(v) + h^1(A).$$
Define the moduli space
$\overline{M}_A$ by the product
$$\overline{M}_A =\prod_{v\in V(A)} \overline{M}_{g(v),n(v)}.$$

There is a canonical way to construct a family of stable $n$-pointed curves of genus 
$g$ over $\overline{M}_A$ using the universal families over each of the factors $\overline{M}_{g(v),n(v)}$. 
This gives the canonical morphism
$$\xi_A: \overline{M}_A \rightarrow \overline{M}_{g,n}$$
whose image is the boundary stratum associated to the graph $A$ and defines a tautological class
$$\xi_{A*}[\overline{M}_A] \in R^*(\overline{M}_{g,n}).$$

A set of additive generators for the tautological ring of $\overline{M}_{g,n}$ is obtained as follows: 
Let $A$ be a stable graph of genus $g$ with $n$ legs. For each vertex $v$ of $A$, let 
$$\theta_v \in R^*(\overline{M}_{g(v),n(v)})$$
be an arbitrary monomial in the $\psi$ and $\kappa$ classes. 
It is proven in \cite{graber_pandharipande_constructions} that these classes give a set of generators for the tautological algebra:

\begin{thm}\label{maude}
 $R^*(\overline{M}_{g,n})$ is generated additively by classes of the form
$$\xi_{A*}\Big(\prod_{v\in V(A)} \theta_v\Big).$$
\end{thm}

By the dimension grading, the list of generators provided
by Theorem \ref{maude} is finite. 
Hence, we obtain the following result.

\begin{cor} \label{fdr}
We have
$\dim_{\mathbb{Q}} \, R^*(\overline{M}_{g,n}) < \infty$.
\end{cor}

\subsection{Gromov-Witten theory and the tautological algebra}
In \cite{faber_pandharipande_relative} Faber and Pandharipande present two geometric methods of constructing tautological classes.
The first method is based on the push-forward of the virtual classes on the moduli spaces of stable maps.
Recall that for a nonsingular projective variety $X$ and a homology class $\beta \in H_2(X, \mathbb{Z})$,
$\overline{M}_{g,n}(X,\beta)$ denotes the moduli space of stable maps representing the class $\beta$.
The natural morphism to the moduli space is denoted by 
$$\rho:\overline{M}_{g,n}(X,\beta) \rightarrow \overline{M}_{g,n},$$
when $2g-2+n>0$. The moduli space $\overline{M}_{g,n}(X,\beta)$
carries a virtual class 
$$[\overline{M}_{g,n}(X,\beta)]^{vir} \in A_*(\overline{M}_{g,n}(X,\beta))$$
obtained from the canonical obstruction theory of maps. 
For a Gromov-Witten class 
$$\omega \in A^*(\overline{M}_{g,n}(X,\beta))$$ 
consider the following push-forward:
$$\rho_* \left( \omega \cap [\overline{M}_{g,n}(X,\beta)]^{vir} \right) \in A_*(\overline{M}_{g,n}).$$
The localization formula for the virtual class \cite{graber_pandharipande_localization} shows that the push-forwards of all Gromov-Witten
classes of compact homogeneous varieties $X$ belong to the tautological ring. 
 
\subsection{The moduli spaces of Hurwitz covers of $\mathbb{P}^1$}
Another tool of geometric construction of tautological classes is using the moduli spaces of Hurwitz 
covers of the projective line. Let $g \geq 0$ and let $\mu^1, \dots , \mu^m$ be $m$ partitions of equal size $d$
satisfying $$2g-2+2d=\sum_{i=1}^m \left( d - \ell (\mu^i) \right),$$
where $\ell(\mu^i)$ denotes the length of the partition $\mu^i$. The moduli space of Hurwitz covers,
$$H_g(\mu^1, \dots , \mu^m)$$ 
parametrizes morphisms 
$$f:C \rightarrow \mathbb{P}^1,$$
where $C$ is a complete, connected, nonsingular curve with marked profiles 
$\mu^1, \dots , \mu^m$ over $m$ ordered points of the target and no ramifications elsewhere. 
The isomorphisms between Hurwitz covers are defined in a natural way. 
The moduli space of admissible covers introduced in \cite{harris_mumford_kodaira} compactifies the moduli space of Hurwitz covers and
contains it as a dense open subset:
$$H_g(\mu^1, \dots , \mu^m) \subset \overline{H}_g(\mu^1, \dots , \mu^m).$$
Let $\rho$ denote the map to the moduli space of curves,
$$\rho: \overline{H}_g(\mu^1, \dots , \mu^m) \rightarrow \overline{M}_{g, \sum_{i=1}^m \ell(\mu^i)}.$$
The push-forward of the fundamental classes of the moduli spaces of admissible covers gives algebraic cycles
on the moduli spaces of curves. 
In \cite{faber_pandharipande_relative} it is proven that these push-forwards are tautological. 
They also study \emph{stable relative maps} which combine features of stable maps and admissible covers. 
For the definition see \cite{faber_pandharipande_relative}. We briefly recall their properties and their connection with the tautological algebras.
There are stable relative maps to a parameterized and unparameterized $\mathbb{P}^1$.
In each case the moduli space is a Deligne-Mumford stack and admits a virtual fundamental class in the expected dimension. There are natural cotangent bundles on these moduli spaces and canonical morphisms to the moduli spaces of curves. 
The notion of \emph{relative Gromov-Witten class} is defined for the moduli space of stable relative maps. 
The compatibility of Gromov-Witten classes on the moduli of stable relative maps and tautological 
classes on the moduli of curves in both the parametrized and unparametrized cases is proven. 
The push-forwards of these classes to the moduli of curves are shown to be tautological. 
As a special case one gets the following result:

\begin{prop}\label{hur}
The moduli of Hurwitz covers yields tautological classes,
$$\rho_*\left( \overline{H}_g(\mu^1, \dots , \mu^m) \right) \in R^*( \overline{M}_{g, \sum_{i=1}^m \ell(\mu^i)}).$$
\end{prop}

An important role of the moduli spaces of admissible covers in the study of tautological classes on the moduli spaces of curves is the basic push-pull method for constructing relations in the tautological algebras.
Consider the two maps to the moduli of curves defined for the moduli space of admissible covers:

$$\label{pushpull}
\begin{CD}
\overline{H}_g(\mu^1 , \dots , \mu^m) @>{\lambda}>>\overline{M}_{g, \sum_{i=1}^n \ell(\mu^i)} \\ 
@V{\pi}VV \\
\overline{M}_{0,m} \\ 
\end{CD}$$

Let $r$ denote a relation among algebraic cycles in $\overline{M}_{0,m}$.
Then $$\lambda_* \pi^*(r),$$
defines a relation in $\overline{M}_{g,  \sum_{i=1}^n \ell(\mu^i)}$.
This method provides a powerful tool to prove algebraic relations in the Chow groups in many cases.
For example, Pandharipande in \cite{pandharipande_geometric} proves that the codimension two relation in $A^2(\overline{M}_{1,4})$,
which is called Getzler's relation, can be obtained by this method.
In \cite{belorousski_pandharipande_descendent} Belorousski and Pandharipande get a codimension two relation in $A^2(\overline{M}_{2,3})$
by considering the pull-back of a relation in $A^2(\overline{M}_{0,9})$ to a suitable moduli space of 
admissible covers and its push-forward to the moduli space $\overline{M}_{2,3}$.
In \cite{faber_pandharipande_relative} Faber and Pandharipande propose the following speculation about the push-pull method:

\begin{spec}
All relations in the tautological ring are obtained via the push-pull method and Proposition \ref{hur}.
\end{spec}

\subsection{The tautological ring of $M_g$}
We saw in Corollary \ref{lam} that the ring $R^*(M_g)$ 
is generated by the $g-2$ classes $\kappa_1, \dots , \kappa_{g-2}$. 
Let us see the explicit relation between the lambda and kappa classes as it is explained in \cite{faber_conjectural}:
\begin{itemize}
\item
Applying the Grothendieck-Riemann-Roch theorem to $\pi:C_g \rightarrow M_g$ and $\omega_{\pi}$ 
relates the Chern character of the Hodge bundle $\mathbb{E}$ in terms of the $\kappa_i$. 
The formula as an identity of formal power series in $t$ is as follows:
$$\sum_{i=0}^{\infty} \lambda_i t^i= \exp \left( \sum_{i=1}^{\infty} \frac{B_{2i} \kappa_{2i-1}}{ 2i(2i-1)} t^{2i-1} \right).$$ 
Here $B_{2i}$ are the Bernoulli numbers. 
This implies that all the $\lambda_i$ can be expressed in the odd $\kappa_i$.  
\item
Using the fact that the relative dualizing sheaf of a nonsingular curve is generated by its global sections, 
Mumford shows that the natural map $\pi^* \mathbb{E} \rightarrow \omega$ of locally free sheaves on $C_g$ is surjective. 
This gives the following vanishing result: 
$$c_j(\pi^* \mathbb{E} - \omega)=0 \qquad \forall j \geq g.$$ 
The desired relation between the lambda's and the kappa's is proven from the analysis of the push-forward of the relations above. 
\end{itemize}

An important property of the tautological ring of $M_g$ is the vanishing result proven by Looijenga \cite{looijenga_tautological}.
He defines the tautological ring of $C_g^n$, the $n$-fold fiber product of the universal curve 
$C_g$ over $M_g$, 
as a subring of $A^*(C_g^n)$ generated by the divisor classes $K_i:=pr_i^* K$, 
which is the pull-back of the class $K=c_1(\omega_{\pi})$ along the projection $pr_i:C_g^n \rightarrow C_g$ to 
the $i^{th}$ factor, and the classes $D_{i,j}$ of the diagonals $x_i=x_j$
and the pull-backs from $M_g$ of the $\kappa_i$. 
He proves that the tautological groups vanish in degrees greater than $g-2+n$ and in degree $g-2+n$ 
they are at most one-dimensional and generated by the class of the locus 
$$H_g^n=\{ (C;x_1, \dots , x_n): C \ \text{hyperelliptic}; x_1= \dots =x_n=x,\ \text{a Weierstrass point}   \}.$$

The first implication of his result is the following:

\begin{thm}
$R^j(M_g)=0$ for all $j > g-2$ and $R^{g-2}(M_g)$ is at most one-dimensional, 
generated by the class of the locus of hyperelliptic curves. 
\end{thm}

According to this vanishing result it is natural to ask whether the tautological ring has dimension one in top degree.
One approach in proving the one-dimensionality of the tautological group $R^{g-2}(M_g)$ is to show that the hyperelliptic locus $H_g$ is non-zero. This was observed by Faber in genus 3 due to the existence of complete curves in $M_3$ and in genus 4 by means of a calculation with test surfaces in 
$\overline{M}_4$. The following result gives the first proof of the one-dimensionality of $R^{g-2}(M_g)$ for every genus $g \geq 2$.   

\begin{thm}
The class $\kappa_{g-2}$ is non-zero on $M_g$. 
Hence $R^{g-2}(M_g)$ is one-dimensional. 
\end{thm}

To prove this Faber considers the classes $\kappa_i$ and $\lambda_i$ on the Deligne-Mumford compactification 
$\overline{M}_g$ of $M_g$. To relate the classes on $M_g$ to the classes on the compactified space he 
observes that the class $ch_{2g-1}(\mathbb{E})$ vanishes on the boundary $\overline{M}_g-M_g$ of the moduli space. 
Then he shows that the following identity on $\overline{M}_g$ holds:
$$\kappa_{g-2} \lambda_{g-1} \lambda_g=\frac{|B_{2g}|(g-1)!}{2^g (2g)!}.$$
The desired result follows since the Bernoulli number $B_{2g}$ does not vanish. 
Note that $\lambda_{g-1} \lambda_g$ is a multiple of $ch_{2g-1}(\mathbb{E})$.

\begin{rem} While the argument above gives a proof of the one-dimensionality of the 
tautological group in degree $g-2$, 
it is interesting to compute the class of the hyperelliptic locus in the tautological algebra. 
In  \cite{faber_pandharipande_logarithmic} Faber and Pandharipande obtain the following formula:
$$[H_g]=\frac{(2^{2g}-1) 2^{g-2}}{(2g+1)(g+1)!} \kappa_{g-2}.$$
\end{rem}

In \cite{faber_conjectural} Faber formulates several important conjectures about the structure of the tautological ring $R^*(M_g)$.
One of these conjectures says that the tautological algebras enjoy an interesting duality:

\begin{conj}\label{mg}
\begin{enumerate}

\item[(a)] The tautological ring $R^*(M_g)$ is Gorenstein with socle in degree $g-2$. 
For a fixed isomorphism $R^{g-2}(M_g)=\mathbb{Q}$ the natural pairing 
$$R^i(M_g) \times R^{g-2-i}(M_g) \rightarrow R^{g-2}(M_g) = \mathbb{Q}$$ is perfect.
\item[(b)] $R^*(M_g)$ behaves like the algebraic cohomology ring of a nonsingular projective variety of dimension $g-2$; 
i.e., it satisfies the Hard Lefschetz and Hodge Positivity properties with respect to the class $\kappa_1$. 
\end{enumerate}
\end{conj}

In \cite{faber_conjectural} it is also conjectured that the $[g/3]$ classes $\kappa_1, \dots , \kappa_{[g/3]}$ generate the ring, 
with no relation in degrees $\leq [g/3]$. 
An explicit conjectural formulas for the proportionalities in degree $g-2$ is also given. 
The following result is due to Faber:
\begin{thm}
The conjecture \ref{mg} is true for $g < 24$.
\end{thm} 

Faber gives the following interesting geometric method to obtain an important class of tautological relations:
First, he introduces certain sheaves on the spaces $C_g^d$. 
Consider the projection $\pi=\pi_{\{1, \dots , d\}}: C_g^{d+1} \rightarrow C_g^d$ that forgets the $(d+1)$-st
point. The sum of the $d$ divisors $D_{1,d+1}, \dots , D_{d,d+1}$ is denoted by $\Delta_{d+1}$:
$$\Delta_{d+1}=D_{1,d+1} + \dots + D_{d,d+1}.$$

The pull-back of $\omega$ on $C_g$ to $C_g^n$ via the projection to the $i^{th}$ factor is denoted by $\omega_i$
and we write $K_i$ for its class in the Chow group. The coherent sheaf $\mathbb{F}_d$ on $C_g^d$ is defined by the formula:
$$\mathbb{F}_d=\pi_*(\mathcal{O}_{\Delta_{d+1}} \otimes  \omega_{d+1}).$$

The sheaf $\mathbb{F}_d$ is locally free of rank $d$; its fiber at a point 
$(C; x_1 , \dots , x_d)=(C;D)$ is the vector space 
$$H^0(C,K/K(-D)).$$

Using Grothendieck-Riemann-Roch one finds the following formula for the total Chern class of 
$\mathbb{F}_d$:
$$c(\mathbb{F_d})=(1+K_1) (1+K_2-\Delta_2) \dots (1+K_d-\Delta_d).$$

The natural evaluation map of locally free sheaves on $C_g^d$ defines the morphism
$$\phi_d: \mathbb{E} \rightarrow \mathbb{F}_d.$$
The kernel of $\phi_d$ over $(C;D)$ is the vector space $H^0(C,K(-D))$.
Faber shows that the following relation holds in the tautological ring of $M_g$:

\begin{prop}
$c_g(\mathbb{F}_{2g-1}-\mathbb{E})=0.$
\end{prop}

This essentially follows since on a curve of genus $g$ the linear system defined by the divisor class $K-D$ is the empty set when $K$ is the canonical divisor on the curve and $D$ is a divisor of degree $2g-1$.
In fact, there are relations of this type for every $d \geq 2g-1$:

\begin{prop}
For all $d \geq 2g-1$, for all $j \geq d-g+1, \qquad c_j(\mathbb{F}_d-\mathbb{E})=0$.
\end{prop}

These give relations between the tautological classes on $C_g^d$ and lead to relations among the generators of the tautological ring of $M_g$.
The method is to multiply these relations with a monomial in the
$K_i$ and $D_{ij}$ and push down to $M_g$.
It is interesting that this method gives the entire ideal of relations in the tautological ring for $g \leq 23$.
Faber makes the following conjecture:
 
\begin{conj}
In the polynomial ring $\mathbb{Q}[\kappa_1, \dots , \kappa_{g-2}]$, 
let $I_g$ be the ideal generated by the relations
of the form $$\pi_*(M \cdot c_j(\mathbb{F}_{2g-1}- \mathbb{E})),$$
with $j \geq g$ and $M$ a monomial in the $K_i$ and $D_{ij}$ and 
$\pi: C_g^{2g-1} \rightarrow M_g$ the forgetful map. 
Then the quotient ring $\mathbb{Q}[\kappa_1 , \dots , \kappa_{g-2} ]/I_g$ is Gorenstein with socle in degree $g-2$; 
hence it is isomorphic to the tautological ring $R^*(M_g)$.
\end{conj}

\subsection*{Relations from higher jets of differentials}

The degree $g+1$ relation is a special case of a larger class of tautological relations.
They are obtained using a method introduced by Faber in \cite{faber_conjectural}.
These relations hold on $\CC_g^n$, where $\pi: \CC_g \too M_g$ is the universal curve over $M_g$.
These relations are based on 2 bundles on $\CC_g^n$ depending on parameters $m,n$.
The first bundle is 

$$\E_m:=\pi_*(\w_{\pi}^{\otimes m}),$$ 

which is the usual Hodge bundle of rank $g$ when $m=1$ and is of rank $(2m-1)(g-1)$ when $m>1$.
The second bundle $\F_{m,n}$ is of rank $n$ defined as follows:

$$\F_{m,n}:=\pi_{n+1,*}(\OO_{\Delta_{n+1}} \otimes \w_{\pi}^{\otimes m}),$$

where $\pi_{n+1}: \CC_g^{n+1} \too \CC_g^n$ forgets the last point and 

$$\Delta_{n+1}:=\sum_{i=1}^n d_{i,n+1}$$ 
is the sum of diagonals.
For such pairs we define $r:=n-2g+2$ when $m=1$ and $r:=n-2m(g-1)$ when $m>1$.
The relation has the following form:

$$c_{g+r-\delta_{m,1}}(\F_{m,n}-\E_m)=0$$

for all $r>0$.

\subsubsection{Faber-Zagier relations}
There is another class of relations in the tautological ring of $M_g$ which was discovered by Faber and Zagier
in their study of the Gorenstein quotient of $R^*(M_g)$. 
To explain their method we need to introduce a generating function. 
Let $$\bold{p}=\{p_1, p_3,p_4,p_6,p_7,p_9,p_{10}, \dots \}$$
be a variable set indexed by the positive integers not congruent to 2 mod 3. 
The formal power series $\Psi$ is defined by the formula:
$$\Psi(t, \bold{p})=(1+t p_3+t^2 p_6 + t^3 p_9 + \dots ) \cdot \sum_{i=0} ^{\infty} \frac{(6i)!}{(3i)!(2i)!} t^i$$
$$\qquad \qquad + (p_1+t p_4 + t^2 p_7 + \dots ) \cdot \sum_{i=0} ^{\infty} \frac{(6i)!}{(3i)!(2i)!} \frac{6i+1}{6i-1} t^i.$$
Let $\sigma$ be a partition of $|\sigma|$ with parts not congruent to 2 modulo 3. 
For such partitions the rational numbers $C_r(\sigma)$ are defined as follows: 
$$\log(\Psi(t,\bold{p}))= \sum_{\sigma} \sum_{r=0}^{\infty} C_r(\sigma) t^r \bold{p}^{\sigma},$$
where $\bold{p}^{\sigma}$ denotes the monomial 
$p_1^{a_1} p_3^{a_3} p_4^{a_4} \dots$ if $\sigma$ is the partition 
$[1^{a_1} 3^{a_3} 4^{a_4} \dots]$. 
Define $$\gamma:=\sum_{\sigma} \sum_{r=0} ^{\infty} C_r(\sigma) \kappa_r t^r \bold{p}^{\sigma};$$ 
then the relation 

\begin{equation}\label{fz}\tag{1}
[\exp(-\gamma)]_{t^r \bold{p}^{\sigma}}=0
\end{equation}
holds in the Gorenstein quotient when 
$g-1+|\sigma| <3r $ and $g \equiv r+|\sigma|+1$ (mod 2). 
For $g \leq 23$ the diagonal relations found by the method of Faber in 
\cite{faber_conjectural} coincide with the Faber-Zagier relations. 
In particular, these are \emph{true relations} and give a complete description of 
the tautological rings for $g \leq 23$. 
We will see below that these relations always hold in the tautological algebra.
It is not known in higher genera whether the Faber-Zagier relations give all relations.
 
\subsubsection{Tautological relations via stable quotients}
In \cite{marian_oprea_pandharipande_SQ} the moduli of stable quotients is introduced and studied. 
These spaces give rise to a class of tautological relations in $R^*(M_g)$ which are called stable quotient relations. 
These relations are based on the function:
$$\Phi(t,x)=\sum_{d=0}^{\infty} \prod_{i=1}^d \frac{1}{1-it} \frac{(-x)^d }{d! t^d}.$$
Define the coefficient $C_d^r$ by the logarithm, 
$$\log(\Phi)= \sum_{d=1}^{\infty} \sum_{r=-1}^{\infty} C_d^r t^r \frac{x^d}{d!}.$$
Let $$\gamma=\sum_{i \geq 1} \frac{B_{2i}}{2i(2i-1)} \kappa_{2i-1} t^{2i-1} +  \sum_{d=1}^{\infty} \sum_{r=-1}^{\infty} C_d^r \kappa_r t^r \frac{x^d}{d!}.$$
The coefficient of $t^r x^d$ in $\exp(-\gamma)$ is denoted by $[\exp(-\gamma)]_{t^r x^d}$, which is an element
of $\mathbb{Q}[\kappa_{-1} , \kappa_0, \kappa_1, \dots]$. 
Recall that 
$$\kappa_{-1}=0, \qquad \kappa_0=2g-2.$$ 
The first class of stable quotient relations is given by the following result.

\begin{thm}
In $R^r(M_g)$, the relation $$[\exp(-\gamma)]_{t^r x^d}=0$$
holds when $g-2d-1 < r$ and $g \equiv r+1$ mod 2.
\end{thm}

In \cite{pandharipande_pixton_relations} the connection between the Faber-Zagier relations and stable quotient relations is studied. 
It is proven that stable quotient relations are equivalent to Faber-Zagier relations.
These give a complete description of the ideal of relations for $g \leq 23$, where the tautological algebra 
of $M_g$ is known to be Gorenstein. The case $g=24$ is unknown. 
In this case the Faber-Zagier relations do not give a Gorenstein algebra. 
According to the result in \cite{pandharipande_pixton_relations} the tautological ring of $M_{24}$ is not Gorenstein or there are 
tautological relations not of the form \eqref{fz}.

\subsection{Filtration of the moduli space} 
The complete analysis of the tautological ring of the moduli space $\overline{M}_{g,n}$ is a difficult question. 
In \cite{faber_hodge} it is proposed to find natural ways to forget some of the boundary strata. 
There is a moduli filtration $$\overline{M}_{g,n} \supset M_{g,n}^{ct} \supset M_{g,n}^{rt} \supset X_{g,n}^{rt},$$
where $X_{g,n}^{rt}$ denotes the reduced fiber of the projection $\pi:\overline{M}_{g,n} \rightarrow \overline{M}_g$
over the moduli point of a smooth curve $X$ of genus $g$.
The associated restriction sequence is 
$$A^*(\overline{M}_{g,n}) \rightarrow A^*(M_{g,n}^{ct}) \rightarrow A^*(M_{g,n}^{rt}) \rightarrow A^*(X_{g,n}^{rt}) \rightarrow 0$$

$$R^*(\overline{M}_{g,n}) \rightarrow R^*(M_{g,n}^{ct}) \rightarrow R^*(M_{g,n}^{rt}) \rightarrow R^*(X_{g,n}^{rt}) \rightarrow 0$$

\begin{rem}
For the open moduli space $M_g \subset \overline{M}_g$ there is a restriction sequence 
$$R^*(\partial \overline{M}_g) \rightarrow R^*(\overline{M}_g) \rightarrow R^*(M_g) \rightarrow 0.$$
It is not known whether the sequence is exact in the middle. 
In \cite{faber_pandharipande_relative} the restriction sequence is conjectured to be exact in all degrees. 
This is proven to be true for $g<24$ in \cite{pandharipande_pixton_relations}. The exactness of the tautological sequences
$$R^*(\overline{M}_{g,n} \backslash M_{g,n}^{ct}) \rightarrow R^*(\overline{M}_{g,n}) \rightarrow R^*(M_{g,n}^{ct}) \rightarrow 0,$$
$$R^*(\overline{M}_{g,n} \backslash M_{g,n}^{rt}) \rightarrow R^*(\overline{M}_{g,n}) \rightarrow R^*(M_{g,n}^{rt}) \rightarrow 0,$$
associated to the compact type and rational tail spaces is conjectured in \cite{faber_pandharipande_relative}.
\end{rem}

\subsection{Evaluations}
Each quotient ring admits a nontrivial linear evaluation $\epsilon$ to $\mathbb{Q}$ obtained by integration. 
The class $\lambda_g$ vanishes when restricted to $\Delta_{irr}$. 
This gives rise to an evaluation $\epsilon$ on $A^*(M_{g,n}^{ct}):$ 
$$\xi \mapsto \epsilon( \xi)=\int_{\overline{M}_{g,n}} \xi \cdot \lambda_g.$$ 

The non-triviality of the $\epsilon$ evaluation is proven by explicit integral computations. 
The following formula for $\lambda_g$ integrals is proven in \cite{faber_pandharipande_hodge}:

$$\int_{\overline{M}_{g,n}} \psi_1^{\alpha_1} \dots \psi_n^{\alpha_n} \lambda_g= \binom{2g-3+n}{\alpha_1, \dots ,\alpha_n} \int_{\overline{M}_{g,1}} \psi_1^{2g-2} \lambda_g.$$

The integrals on the right side are evaluated in terms of the Bernoulli numbers: 
$$\int_{\overline{M}_{g,1}} \psi_1^{2g-2} \lambda_g=\frac{2^{2g-1}-1}{2^{2g-1}}\frac{|B_{2g}|}{(2g)!}.$$
This proves the non-triviality of the evaluation since $B_{2g}$ doesn't vanish. 
It is proven in \cite{faber_algorithms} that for $g>0$ the class $\lambda_{g-1} \lambda_g$ 
vanishes when restricted to the complement of the open subset $M_{g,n}^{rt}$. 
This leads to an evaluation $\epsilon$ on $A^*(M_{g,n}^{rt})$:
$$\xi \mapsto \epsilon(\xi)=\int_{\overline{M}_{g,n}} \xi \cdot \lambda_{g-1} \lambda_{g}.$$

$$\int_{\overline{M}_{g,n}} \psi_1^{\alpha_1} \dots \psi_n^{\alpha_n} \lambda_{g-1} \lambda_{g}= \frac{(2g+n-3)! (2g-1)!!}{(2g-1)! \prod_{i=1}^n (2 \alpha_i -1)!!} \int_{\overline{M}_{g,1}} \psi_1^{g-1} \lambda_{g-1} \lambda_{g},$$
where $g \geq 2$ and $\alpha_i \geq 1$. 
In \cite{getzler_pandharipande_virasoro} it is shown that the degree zero Virasoro conjecture applied to $\mathbb{P}^2$ implies this prediction. 
The constant 
$$\int_{\overline{M}_{g,1}} \psi_1^{g-1} \lambda_{g-1} \lambda_{g}=\frac{1}{2^{2g-1}(2g-1)!!}\frac{|B_{2g}|}{2g}$$
has been calculated by Faber, who shows that it follows from Witten's conjecture.

\subsection{Known facts}
The following results are known:
\begin{itemize}
\item[(a)] $R^*(M_{g,n}^{rt})$ vanishes in degrees $> g-2+n-\delta_{0g}$ and is 1-dimensional in degree 
$g-2+n-\delta_{0g}$.
\item[(b)] $R^*(M_{g,n}^{ct})$ vanishes in degrees $> 2g-3+n$ and is 1-dimensional in degree $2g-3+n$.
\item[(c)] $R^*(\overline{M}_{g,n})$ vanishes in degrees $> 3g-3+n$ and is 1-dimensional in degree $3g-3+n$.
\end{itemize}
The statement (a) is due to Looijenga \cite{looijenga_tautological}, Faber \cite{faber_conjectural}, Faber and Pandharipande \cite{faber_pandharipande_logarithmic}.
Graber and Vakil in \cite{graber_vakil_tautological,graber_vakil_localization} proved (b),(c).
In their study of relative maps and tautological classes Faber and Pandharipande \cite{faber_pandharipande_relative}
give another proof of (b),(c).

\subsection{Gorenstein conjectures and their failure}

We have discussed conjectures of Faber on the tautological ring of $M_g$.
Analogue conjectures were formulated by Faber and Pandharipande for pointed spaces and their compactification:

\begin{enumerate}
\item[(A)] $R^*(M_{g,n}^{rt})$ is Gorenstein with socle in degree  $g-2+n-\delta_{0g}$.

\item[(B)] $R^*(M_{g,n}^{ct})$ is Gorenstein with socle in degree $2g-3+n$.

\item[(C)] $R^*(\overline{M}_{g,n})$ is Gorenstein with socle in degree $3g-3+n$.
\end{enumerate}

Hain and Looijenga introduce a compactly supported version of the tautological algebra: 
The algebra $R^*_c(M_{g,n})$ is defined to be the set of elements in 
$R^*(\overline{M}_{g,n})$ that restrict trivially to the Deligne-Mumford boundary. 
This is a graded ideal in $R^*(\overline{M}_{g,n})$ and the intersection product defines a map 
$$R^*(M_{g,n}) \times R^{*}_c(M_{g,n}) \rightarrow R^{*}_c(M_{g,n})$$ 
that makes $R^*_c(M_{g,n})$ a $R^*(M_{g,n})$-module. 
In \cite{hain_looijenga_mapping} they formulated the following conjecture for the case $n=0$:

\begin{conj}\label{hl}
\begin{enumerate}
\item[(1)] The intersection pairings 
$$R^k(M_g) \times R^{3g-3-k}_c(M_g) \rightarrow R_c^{3g-3}(M_g) \cong \mathbb{Q}$$ 
are perfect for $k \geq 0$.

\item[(2)] In addition to (1), $R^*_c(M_g)$ is a free $R^*(M_g)$-module of rank one. 
\end{enumerate}
\end{conj}

There is a generalization of the notion of the compactly supported tautological algebra to the space 
$M_{g,n}^{rt}$: In \cite{faber_remark} Faber defines $R^*_c(M_{g,n}^{rt})$ as the set of elements in 
$R^*(\overline{M}_{g,n})$
that restrict trivially to $\overline{M}_{g,n} \backslash M_{g,n}^{rt}$. 
He considers the following generalization of the conjectures above:

\begin{conj} \label{HL}
\begin{enumerate}
\item[(D)] The intersection pairings 
$$R^k(M_{g,n}^{rt}) \times R^{3g-3+n-k}_c(M_{g,n}^{rt}) \rightarrow R_c^{3g-3+n}(M_{g,n}^{rt}) \cong \mathbb{Q}$$ are perfect for $k \geq 0.$

\item[(E)] In addition to D, $R^*_c(M_{g,n}^{rt})$ is a free $R^*(M_{g,n}^{rt})$-module of rank one. 
\end{enumerate}
\end{conj}

The relation between the Gorenstein conjectures and the conjectures of Hain and Looijenga is discussed in \cite{faber_remark}. To state the result let us define a partial ordering on the set of pairs $(g,n)$ of nonnegatvie integers such that $2g-2+n>0$. We say that $(h,m) \leq (g,n)$ if and only if $h \leq g$ and $2h-2+m \leq 2g-2+n$.
This is equivalent to saying that there exists a stable curve of genus $g$ whose dual graph contains a vertex of genus $h$ with valency $m$. Faber proves the following fact:

\begin{thm}\label{rt}
Conjectures (A) and (C) are true for all $(g,n)$ if and only if conjecture (E) is true for all $(g,n)$. More precisely, 
$$A_{(g,n)} \  \text{and} \ C_{(g,n)} \Rightarrow E_{(g,n)} \Rightarrow A_{(g,n)} \ \text{and} \ D_{(g,n)}$$
and $$\{ D_{(g^{'},n^{'})} \}_{(g^{'} , n^{'}) \leq (g,n)} \Rightarrow \{C_{(g^{'}, n^{'})}\}_{(g^{'}, n^{'}) \leq (g,n)}.$$
\end{thm}

In a similar way Faber defines $R^*_c(M_{g,n}^{ct})$ as the set of elements in 
$R^*(\overline{M}_{g,n})$ that pull back
to zero via the standard map $\overline{M}_{g-1,n+2} \rightarrow \overline{M}_{g,n}$ onto $\Delta_{irr}$. 
By considering the $(D^{ct})$ and $(E^{ct})$ analogues to $(D)$ and $(E)$ 
he shows the compact type version of Theorem \ref{rt}, which reads as follows:
$$B_{(g,n)} \  \text{and} \ C_{(g,n)} \Rightarrow E^{ct}_{(g,n)} \Rightarrow B_{(g,n)} \ \text{and} \ D^{ct}_{(g,n)}$$
and $$\{ D^{ct}_{(g^{'},n^{'})} \}_{(g^{'} , n^{'}) \leq (g,n)} \Rightarrow \{C_{(g^{'}, n^{'})}\}_{(g^{'}, n^{'}) \leq (g,n)}.$$

In \cite{tavakol_1} we show that the tautological ring of the moduli space $M_{1,n}^{ct}$ is Gorenstein. 
This was based on the complete analysis of the space of tautological relations:

\begin{thm}
The space of tautological relations on the moduli space $M_{1,n}^{ct}$ is generated by Keel relations in genus zero and Getzler's relation.
In particular, $R^*(M_{1,n}^{ct})$ is a Gorenstein algebra. 
\end{thm}

\begin{proof}
The proof consists of the following parts: 
\begin{itemize}
\item The case of a fixed curve: The tautological ring $R^*(C^n)$, 
for a smooth curve $C$ of genus $g$, was defined by Faber and Pandharipande (unpublished).
They show that the image $RH^*(C^n)$ in cohomology is Gorenstein. 
In \cite{green_griffiths} Green and Griffiths have shown that $R^*(C^2)$ is not Gorenstein, 
for $C$ a generic complex curve of genus $g \geq 4$.
In arbitrary characteristic \cite{yin_fp}. 
The study of the algebra $R^*(C^n)$ shows that it is Gorenstein when $C$ is an elliptic curve. 
It is interesting that there are two essential relations in $R^2(C^3)$ and $R^2(C^4)$ 
which play an important role in the proof of the Gorenstein property of the algebra $R^*(C^n)$. 
These relations are closely related to the relation found by Getzler \cite{getzler_intersection} in $R^2(\overline{M}_{1,4})$. 
In \cite{pandharipande_geometric}, Pandharipande gives a direct construction of Getzler's relation via a rational equivalence 
in the Chow group $A_2(\overline{M}_{1,4})$. 
By using some results from the representation theory of symmetric groups and Brauer's centralizer algebras we prove that 
$R^*(C^n)$ is Gorenstein for $n \geq 1$.

\item The reduced fiber of $M_{1,n}^{ct} \rightarrow M_{1,1}^{ct}$ over $[C] \in M_{1,1}^{ct}$: 
This fiber, which is denoted by $\overline{U}_{n-1}$, is described as a sequence of blow-ups of the variety $C^{n-1}$.
There is a natural way to define the tautological ring $R^*(\overline{U}_{n-1})$ of the fiber $\overline{U}_{n-1}$. 
The analysis of the intersection ring of this blow-up space shows that there is a natural filtration on 
$R^*(\overline{U}_{n-1})$.
As a result we show that the intersection matrices of the pairings for the tautological algebra have a triangular property. 
It follows that $R^*(\overline{U}_{n-1})$ is Gorenstein. 

\item The isomorphism $R^*(M_{1,n}^{ct}) \cong R^*(\overline{U}_{n-1})$: 
For this part we verify directly that the relations predicted by the fiber $\overline{U}_{n-1}$ 
indeed hold on the moduli space $M_{1,n}^{ct}$. 
As a result, we see that the tautological ring $R^*(M_{1,n}^{ct})$ is Gorenstein. 
\end{itemize}

\end{proof}

With very similar methods we can describe the tautological ring of the moduli space $M_{2,n}^{rt}$:

\begin{thm}
The tautological ring of the moduli space $M_{2,n}^{rt}$ is generated by the following relations:

\begin{itemize}
\item Getzler's relation in genus two,

\item Belorousski-Pandharipande relation,

\item A degree 3 relation on $M_{2,6}^{rt}$.

\item Extra relations from the geometry of blow-up.

\end{itemize}

\end{thm}

\begin{proof}
The method of the study of the tautological ring $R^*(M_{2,n}^{rt})$ in \cite{tavakol_2},
is similar to that of $M_{1,n}^{ct}$: 
We first consider a fixed smooth curve $X$ of genus two. 
In this case there are two essential relations in $R^2(X^3)$ and $R^3(X^6)$ to get the 
Gorenstein property of $R^*(X^n)$ for $n \geq 1$.
The relation in $R^2(X^3)$ is closely related to known relations discovered by Faber \cite{faber_conjectural} and 
Belorousski-Pandharipande \cite{belorousski_pandharipande_descendent} on different moduli spaces. 
The relation in $R^3(X^6)$ is proven using the same method as Faber used in \cite{faber_conjectural}. 
To find the relation between $R^*(X^n)$ and the tautological ring of the moduli space $M_{2,n}^{rt}$ 
we consider the reduced fiber of $\pi:\overline{M}_{2,n} \rightarrow \overline{M}_2$ over $[X] \in M_2$, 
which is the Fulton-MacPherson 
compactification $X[n]$ of the configuration space $F(X,n)$. 
There is a natural way to define the tautological ring for 
the space $X[n]$. The study of this algebra shows that it is Gorenstein. 
We finish the proof by showing that there is an isomorphism between the tautological ring of $X[n]$ and $R^*(M_{2,n}^{rt})$. 
\end{proof}

\subsection{Failure}

According to Gorenstein conjectures tautological rings should have a form of Poincar\'e duality.
Computations of Faber for $g <24$ verifies this expectation for $M_g$.
There are more evidences for pointed spaces \cite{petersen_1, tavakol_1, tavakol_2}.
The striking result of Petersen and Tomassi \cite{petersen_tommasi_gorenstein} showed the existence of a counterexample in genus two.
More precisely, they showed that for some $n$ in the set $\{12, 16, 20\}$ the tautological ring of $\M_{2,n}$ is not Gorenstein. 
We know that the moduli space is of dimension $n+3$ and the failure happens for the pairing between degrees $k,k+1$ when $n=2k-2$.
A crucial ingredient in their approach is to give a characterization of tautological classes in cohomology.
This part is based on the results \cite{harder_eisenstein} of Harder on Eisenstein cohomology of local systems on the moduli of abelian surfaces.
In another work \cite{petersen_2} Petersen determined the group structure of the cohomology of $M_{2,n}^{ct}$.
In particular, he shows that $R^*(M_{2,8}^{ct})$ is not Gorenstein.
While there are counterexamples for Gorenstein conjectures the case of $M_{g,n}^{rt}$ is still open.
A very interesting case is the moduli space $M_{24}$ of curves of genus 24.
In this case all known methods do not give a Gorenstein algebra.

%\cite{petersen_tommasi_gorenstein, petersen_2, petersen_d_elliptic, petersen_abelian_surfaces}

\subsection{Pixton's conjecture}

One of the most important applications of the Gorenstein conjectures was that they would determine the ring structure.
After the counterexamples to these conjectures the key question is to give a description of the space of tautological relations.
In \cite{pixton_conjectural} Pixton proposed a class of conjectural tautological relations on the Deligne-Mumford space $\M_{g,n}$.
His relations are defined as weighted sums over all stable graphs on the boundary $\partial M_{g,n}=\M_{g,n} \setminus M_{g,n}$.
Weights on graphs are tautological classes.
Roughly speaking, his relations are the analogue of Faber-Zagier relations which were originally defined on $M_g$.
Pixton also conjectures that all tautological relations have this form.
A recent result \cite{pixton_pandharipande_zvonkine_relations} of Pandharipande, Pixton and Zvonkine shows that Pixton's relations are connected with the Witten's class on the space of curves with 3-spin structure.
Their analysis shows that Pixton's relations hold in cohomology.
Janda \cite{janda_tautological} derives Pixton's relations by applying the virtual localization formula to the moduli space of stable quotients discussed before.
This establishes Pixton's relations in Chow. 
The proof given in \cite{pixton_pandharipande_zvonkine_relations} shows more than establishing Pixton's relations in cohomology. 
It shows that semi-simple cohomological field theories give a rich source of producing tautological relations.
These connections are studied in subsequent articles \cite{janda_comparing, janda_relations}.
In \cite{janda_semisimple} Janda shows that any relation coming from semi-simple theories can be expressed in terms of Pixton's relations. 
Witten class concerns the simplest class of isolated singularities.
A more general construction due to Fan-Jarvis-Ruan \cite{FJRW_virtual, FJRW_quantum}, also known as FJRW theory, provides an analytic method to construct the virtual class for the moduli spaces associated with isolated singularities.
An algebraic approach in constructing the virtual class is given by Polishchuk and Vaintrob \cite{polishchuk_vaintrob_matrix}.
This construction gives a powerful tool to produce tautological relations in more general cases.
See \cite{guere_hodge} for more details. 
Several known relations such as Keel's relation \cite{keel_intersection} on $\M_{0,4}$, 
Getzler's relation \cite{getzler_intersection} on $\M_{1,4}$ and Belorousski-Pandharipande relation \cite{belorousski_pandharipande_descendent} on $\M_{2,3}$ follow from Pixton's relations.
For a recent survey on tautological classes we refer the reader to \cite{pandharipande_calculus}.

\subsection{Tautological relations from the universal Jacobian}

In his recent thesis \cite{yin_thesis}, Yin studies the connection between tautological classes on moduli spaces of curves and the universal Jacobian.
The $\mathfrak{sl}_2$ action on the Chow group of abelian schemes and   
Polishchuk's differential operator give a rich source of tautological relations.

\subsubsection*{Tautological classes on the Jacobian side}

The tautological ring of a fixed Jacobian variety under algebraic equivalence is defined and studied by Beauville \cite{beauville_algebraic}. 
One considers the class of a curve of genus $g$ 
inside its Jacobian and apply all natural operators to it induced from 
the group structure on the Jacobian and the intersection product in the Chow ring. 
The following result is due to Beauville \cite{beauville_algebraic}:

\begin{thm} 
The tautological ring of a Jacobian variety is finitely generated.
Furthermore, it is stable under the Fourier-Mukai transform.
\end{thm}

In fact, if one applies the Fourier transform to the class of the curve, 
all components in different degrees belong to the tautological algebra.
The connection between tautological classes on the universal Jacobian and moduli of curves is studied in the recent thesis of Yin \cite{yin_thesis}.
Let $\pi: \CC \too S$ be a family of smooth curves of genus $g > 0$ which admits a section $s: S \too \CC$. 
Denote by $\J_g:=\text{Pic}^0(\CC/S)$ the relative Picard scheme of divisors of degree zero. 
It is an abelian scheme over the base $S$ of relative dimension $g$.
The section $s$ induces an injection $\iota: \CC \too \J_g$ from $\CC$ into the universal Jacobian $\J_g$.
The geometric point $x$ on a curve $C$ is sent to the line bundle $\OO_C(x-s)$ via the morphism $\iota$. 
The abelian scheme $\J_g$ is equipped with the Beauville decomposition defined in \cite{beauville_algebraic}. 
Components of this decomposition are eigenspaces of the natural maps corresponding to multiplication with integers.
More precisely, for an integer $k$ consider the associated endomorphism on $\J_g$.
The subgroup $A^i_{(j)}(\J_g)$ is defined as all degree $i$ classes on which the morphism $k^*$ acts via multiplication with $k^{2i-j}$.
Equivalently, the action of the morphism $k_*$ on $A^i_{(j)}(\J_g)$ is multiplication by $k^{2g-2i+j}$. 
The Beauville decomposition has the following form:
$$A^*(\J_g)=\oplus_{i,j} A_{(i,j)}(\J_g),$$
where $A_{(i,j)}(\J_g):=A_{(j)}^{\frac{i+j}{2}}(\J_g)$ for $i \equiv j$ mod 2.
The Pontryagin product $x * y$ of two Chow classes $x,y \in A^*(\J_g)$ is defined as $\mu_*(\pi_1^* x \cdot \pi_2^* y)$, where
$$\mu: \J_g \times_S \J_g \too \J_g$$
 
and $\pi_1,\pi_2: \J_g \times_S \J_g \too \J_g$ be the natural projections.
The universal theta divisor $\theta$ trivialized along the zero section 
is defined in the rational Picard group of $\J_g$.
It defines a principal polarization on $\J_g$. 
Let $\mathcal{P}$ be the universal Poincar{\'e} bundle on $\J_g \times_S \J_g$ trivialized along the zero sections.
Here we use the principal polarization to inentify $\J_g$ with its dual.
The first Chern class $l$ of $\mathcal{P}$ is equal to 
$\pi_1^* \theta+\pi_2^* \theta-\mu^* \theta$.
The Fourier Mukai transform $\CF$ gives an isomorphism between $(A^* (\J_g),.)$ and $(A^*(\J_g),*)$.
It is defined as follows:
$$\CF(x)=\pi_{2,*}(\pi_1^* x \cdot \exp(l)).$$ 
We now recall the definition of the tautological ring of $\J_g$ from \cite{yin_thesis}.
It is defined as the smallest $\QQ$-subalgebra of the Chow ring $A^*(\J_g)$ which contains the class of $\CC$ and is stable under the Fourier transform and all maps $k^*$ for integers $k$. 
It follows that for an integer $k$ it becomes stable under $k_*$ as well.
One can see that the tautological algebra is finitely generated.
In particular, it has finite dimensions in each degree.
The generators are expressed in terms of the components of the curve class in the Beauville decomposition.
Define the following classes:

$$p_{i,j}:=\CF \left(\theta^{\frac{j-i+2}{2}} \cdot [\CC]_{(j)} \right) \in A_{(i,j)}(\J_g).$$

We have that $p_{2,0}=-\theta$ and $p_{0,0}=g[\J_g]$. 
The class $p_{i,j}$ vanishes for $i<0$ or $j<0$ or $j>2g-2$.
The tautological class $\Psi$ is defined as $$\Psi:=s^*(K),$$

where $K$ is the canonical class defined in Section 1.
The pull-back of $\Psi$ via the natural map $\J_g \too S$ is denoted by the same letter. 

\subsection{Lefschetz decomposition of Chow groups}
The action of $\mathfrak{sl}_2$ on Chow groups of a fixed abelian variety was studied by K{\"u}nnemann \cite{kunnemann_lefschetz}.
Polishchuk \cite{polishchuk_universal} has studied the $\mathfrak{sl}_2$ action for abelian schemes. 
We follow the standard convention that $\mathfrak{sl}_2$ is generated by elements $e,f,h$ satisfying: 
$$[e,f]=h \D [h,e]=2e, \D [h,f]=-2f.$$
In this notation the action of $\mathfrak{sl}_2$ on Chow groups of $\J_g$ is defined as
$$e: A_{(j)}^i(\J_g) \too A_{(j)}^{i+1}(\J_g) \D x \too -\theta \cdot x,$$
$$f: A_{(j)}^i(\J_g) \too A_{(j)}^{i-1}(\J_g) \D x \too -\frac{\theta^{g-1}}{(g-1)!} * x,$$
$$h: A_{(j)}^i(\J_g) \too A_{(j)}^i(\J_g) \D x \too -(2i-j-g) x,$$

The operator $f$ is given by the following differential operator:

$$\mathcal{D}=\frac{1}{2} \sum_{i,j,k,l} \left( \Psi p_{i-1,j-1}p_{k-1,l-1}- \binom{i+k-2}{i-1}p_{i+k-2,j+l} \right) \partial p_{i,j} \partial p_{k,l}+\sum_{i,j} p_{i-2,j}\partial p_{i,j}.$$

The following fact is proved in \cite{yin_thesis}:

\begin{thm}
The tautological ring of $\J_g$ is generated by the classes $\{p_{i,j}\}$ and $\Psi$.
In particular, it is finitely generated. 
\end{thm}

The differential operator $\mathcal{D}$ is a powerful tool to produce tautological relations. 
We can take any relation and apply the operator $\mathcal{D}$ to it several times. 
This procedure yields a large class of tautological relations.
One can get highly non-trivial relations from this method. 
All tautological relations on the universal curve $\CC_g$ for $g \leq 19$ and on $\CM_g$ for $g \leq 23$ are recovered using this method.
The following conjecture is proposed by Yin:

\begin{conj}
Every relation in the tautological ring of $\CC_g$ comes from a relation on the universal Jacobian $\J_g$.
\end{conj}

Yin further conjectures that the $\mathfrak{sl}_2$ action is the only source of all tautological relations.
For more details and precise statements we refer to Conjecture 3.19. in \cite{yin_thesis}.

%\cite{yin_thesis, tavakol_hyperelliptic}

\subsubsection*{Relations on $C_g^n$ from the universal Jacobian}

The work of Yin \cite{yin_thesis} shows that the analysis of tautological classes on the Jacobian side is a powerful tool in the study of tautological relations.
A natural question is whether one can prove relations in more general cases.
One can define many maps from products of the universal curve into the universal Jacobian.
The pull-back of tautological relations give a rich source of relations on moduli spaces of curves.
However, these relations do not seem to be sufficient to produce all relations.
A crucial problem is that they have a symmetric nature and one needs to find an effective way to break the symmetry.
The first example happens in genus 3 with 5 points.
In this case there is a \emph{symmetric} relation of degree 3 which is not the pull-back of any relation from the Jacobian side.
Here we describe a simple method which can be a candidate for breaking the symmetry.
Let $S=\CM_{g,1}$ and consider the universal curve $\pi:\CC \too S$ together with the natural section $s:S \too \CC$.
The image of the section $s$ defines a divisor class $x$ on $\CC$.
For a natural number $n$ consider the $n$-fold fibred product $\CC^n$ of $\CC$ over $S$.
Consider the map

$$\phi_n:\CC^n \too \J_g$$

which sends a moduli point $(C, p_1, \dots, p_n)$ to $\sum_{i=1}^n p_i-n \cdot x$.
In \cite{yin_thesis} Yin describes an algorithm to compute the pull-back of tautological classes on $\J_g$ to $\CC^n$ via these maps.
Let $\CR_n$ be the polynomial algebra generated by tautological classes on $\CC^n$. 
We define the ideal $\I_n \subset \CR_n$ of tautological relations as follows:
For each relation $R$ on the Jacobian side and an integer $m \geq n$ consider the pull-back $\phi_m^*(R)$ on $\CC^m$.
We can intersect $\phi_m^*(R)$ with a tautological class and push it forward to $\CC^k$ for $n \leq k \leq m$.
This procedure yields many tautological relations in $R^*(\CC^n)$.
Denote the resulting space of relations by $\I_n$.
The following seems to be a natural analogue of Yin's conjecture for $\CC^n$:

\begin{conj}\label{Jacobian}
The natural map $\frac{\CR_n}{\I_n} \too R^*(\CC^n)$ is an isomorphism.
\end{conj}

Let $\pi: \CC_g \too \CM_g$ be the universal curve and consider the space $\CC_g^n$ for a natural number $n$.
Notice that there is a natural isomorphism between $\CC^n$ and $\CC_g^{n+1}$.
Under this isomorphism we get the following identification

$$\al_n: R^*(\CC^n) \cong R^*(\CC_g^{n+1}).$$

%Using the classes mentioned above on $\CC^n$ we obtain tautological classes on $\CC_g^{n+1}$ via the following rules:

%$$\psi \too K_{n+1}, \D x_i \too d_{i,n+1} \D 1 \leq i \leq n.$$

%This gives the following map:

%$$\al_n: R^*(\CC^n) \too R^*(\CC_g^{n+1})$$

%\begin{conj}\label{motive}
%The map $\al_n$ is an isomorphism between the tautological rings.
%\end{conj}

Conjecture \ref{Jacobian} does not lead to an effective way to describe tautological relations.
The reason is that for a given genus $g$ and a natural number $n$ one gets infinitely many relations.
However, for a given $m \geq n$ there are only a finite number of motivic relations.
An optimistic speculation is that there is a uniform bound independent of the number of points which only depends on the genus: 

\begin{conj}\label{UC}
For a given genus $g$ there is a natural number $N_g$ such that all tautological relations on $C_g^n$ lie in the image of $\phi_m$ for $m \leq n+N_g$.
\end{conj}

\begin{thm}\label{genus2}
Conjecture \ref{UC} is true for $g=2$.
\end{thm}

\begin{proof}
Here we only sketch the proof since the computations are a bit involved and are done using computer. 
From \cite{tavakol_2} we know that there are 3 basic relations which generate the space of tautological relations on products of the universal curve of genus two.
We need to show that these relations can be obtained from the Jacobian side by the method described above.
It turns out that the motivic relation $p_{3,1}^2=0$ on the Jacobian side gives several interesting relations on the curve side.
One can show that this vanishing gives the relation $\ka_1=0$ on the moduli space $M_2$.
More precisely, after applying the Polishchuk differential operator $\DD$ to this class 3 times we get the degree one relation $\ka_1=0$.
To prove the vanishing of the Faber-Pandharipande cycle we need to consider the class $\DD^2(p_{3,1}^2)$ and compute its pull-back to $\CC$.
The pull-back of the same class $\DD^2(p_{3,1}^2)$ to $\CC^2$ gives the vanishing of the Gross-Schoen cycle.
Finally, the degree 3 relation follows from the vanishing of $p_{2,0}^3$.
\end{proof}

\begin{rem}
The proof of Theorem \ref{genus2} shows a stronger statement than Conjecture \ref{UC}.
It shows that any relation on $\CC_2^n$ is simply the pull-back of a motivic relation from the Jacobian side.
In other words, we find that $N_2=0$.
That is, of course, not expected to be true in general. 
From the results proved in \cite{POWR} we get a proof of Conjecture \ref{UC} in genus 3 and 4. Furthermore, it is shown that $N_3, N_4>0$.
\end{rem}

\subsection{Curves with special linear systems}

The theory of Yin is quite powerful also in the study of special curves.
The simplest example deals with hyperellpitic curves.
In this case we have determined the structure of the tautological ring on the space $H_{g,n}^{rt}$ of hyperelliptic curves with rational tails \cite{tavakol_hyperelliptic}.
The crucial part is to understand the space of relations on products of the universal hyperelliptic curve:

\begin{thm}
The space of tautological relations on products of the universal hyperelliptic curve is generated by the following relations:

\begin{itemize}
\item
The vanishing of the Faber-Pandharipande cycle,

\item
The vanishing of the Gross-Schoen cycle,

\item
A symmetric relation of degree $g+1$ involving $2g+2$ points.
\end{itemize}
\end{thm}

\begin{proof}
The idea is very similar to the proof of Theorem \ref{genus2}.
Here we can show that all of these relations follow from known relations on the Jacobian side.
The difference is that we do not consider tautological classes on the Jacobian varieties associated with generic curves.
For technical reasons we need to work over the moduli space of Weierstrass pointed hyperelliptic curves.
Denote by $\J_g$ the universal Jacobian over this moduli space.
The basic relation is that all components $C_{(j)}$ vanish when $j>0$.
Using these vanishings one can show that both classes in the first and second part are zero.
The last relation is shown to follow from the vanishing of the well-known relation $\theta^{g+1}$.
This connection also shows that the question of extending tautological relations to the space of curves of compact types is equivalent to extending the first two classes.
\end{proof}

A more general case of curves with special linear systems can be treated with this method.
Based on a series of relations on such curves we have formulated a conjectural description of the space of tautological relations
for certain classes of trigonal and 4-gonal curves. Details will be discussed in upcoming articles. 

\part*{Cohomology of moduli spaces of curves}

Tautological classes on the moduli spaces of curves are natural algebraic cycles which define cohomology classes via the cycle class map. 
It would be interesting to understand the image and determine to what extent the 
even cohomology groups are generated by tautological classes. 
It is proven by Arbarello and Cornalba \cite{arbarello_cornalba_calculating} that $H^k(\overline{M}_{g,n},\mathbb{Q})$ 
is zero for $k=1,3,5$ and all values of $g$ and $n$
such that these spaces are defined. They also prove that the second cohomology group 
$H^2(\overline{M}_{g,n}, \mathbb{Q})$ is generated 
by tautological classes, modulo explicit relations. 
To state their result let us recall the relevant notations:
Let $0 \leq a \leq g$ be an integer and $S \subset \{1, \dots , n\}$ be a subset such that
$2a-2+|S| \geq 0$ and $2g-2a-2+|S^c| \geq 0$. The divisor class $\delta_{a,S}$ stands 
for the $Q$-class in $\overline{M}_{g,n}$ whose generic point represents a nodal curve consisting 
of two irreducible components of genera $a,g-a$ and markings on the component of genus $a$ are labeled
by the set $S$ while markings on the other component are labeled by the complement $S^c$.
The sum of all classes $\delta_{a,S}$ is denoted by $\delta_{a}$.
In case $g=2a$ the summand $\delta_{a,S}=\delta_{a,S^c}$ occurs only once in this sum. 
The divisor class $\delta_{irr}$ is the $Q$-class of the image of the gluing morphism 
$$\iota: \overline{M}_{g-1, n \cup \{* , \bullet\}} \rightarrow \overline{M}_{g,n}.$$
The result is the following:

\begin{thm}\label{coh}
For any $g$ and $n$ such that $2g-2+n>0$, $H^2(\overline{M}_{g,n})$ is generated by the classes 
$\kappa_1, \psi_1, \dots , \psi_n, \delta_{irr}$, and the classes $\delta_{a,S}$ such that $0 \leq a \leq g$, 
$2a-2+|S| \geq 0$ and $2g-2a-2+|S^c| \geq 0$. The relations among these classes are as follows: 
\begin{itemize}
\item If $g>2$ all relations are generated by those of the form 
\begin{equation}\label{1}\tag{2}
\delta_{a,S}=\delta_{g-a,S^c}.
\end{equation}

\item If $g=2$ all relations are generated by \eqref{1} plus the following one
$$5 \kappa_1=5\sum_{i=1}^n \psi_i + \delta_{irr}-5 \delta_0+7 \delta_1.$$
\item If $g=1$ all relations are generated by \eqref{1} plus the following ones $$\kappa_1=\sum_{i=1}^n \psi_i -\delta_0,$$
$$12 \psi_p=\delta_{irr} +12 \sum_{p \in S, |S| \geq 2} \delta_{0,S}.$$

\item  If $g=0$, all relations are generated by \eqref{1} and the following ones 
$$\kappa_1=\sum (|S|-1) \delta_{0,S},$$
$$\psi_z=\sum_{z \in S; x,y \notin S} \delta_{0,S}, \qquad \delta_{irr}=0.$$

\end{itemize}
\end{thm}

In \cite{edidin_codimension} Edidin shows that the fourth cohomology group 
$H^4(\overline{M}_g)$ is tautological for $g \geq 12$.
Polito \cite{polito_fourth} gives a complete description of the fourth tautological group of $\overline{M}_{g,n}$ 
and proves that for $g \geq 8$ it coincides with the cohomology group. 
Deligne shows in \cite{deligne_constantes} that the cohomology group $H^{11,0}(\overline{M}_{1,11})$ is non-zero. 
It is not known whether the cohomology groups $H^k(\overline{M}_{g,n})$ for $k=7,9$ vanish.
The following vanishing result is due to Harer \cite{harer_virtual}:
\begin{thm}
The moduli space $M_{g,n}$ has the homotopy type of a finite cell-complex of dimension $4g-4+n$, $n>0$. 
It follows that 
$$H_k(M_{g,n},\mathbb{Z})=0$$ if $n>0$ and $k > 4g-4+n$ and 
$$H_k(M_g,\mathbb{Q})=0$$ if $k>4g-5$.
\end{thm}

Using the method of counting points over finite fields Bergstr\"om has computed the equivariant Hodge Euler characteristic of the moduli spaces of
$\overline{M}_{2,n}$ for $4 \leq n \leq 7$ \cite{bergstrom_hyperelliptic} and that of $\overline{M}_{3,n}$ for $2 \leq n \leq 5$ \cite{bergstrom_3}.  
In \cite{bergstrom_tommasi_rational} Bergstr\"om and Tommasi determine the rational cohomology of $\overline{M}_4$. 
A complete conjectural description of the cohomology of $\overline{M}_{2,n}$ is given in \cite{faber_geer_genus2_1, faber_geer_genus2_2}.

\subsection{The Euler characteristics of the moduli spaces of curves}
In \cite{harer_zagier_euler} Harer and Zagier computed the orbifold Euler characteristics of $\mathcal{M}_{g,n}$ to be 
$$\chi(\mathcal{M}_{g,n})=(-1)^n \frac{(2g+n-3)!}{2g(2g-2)!} B_{2g}$$
if $g>0$. For $g=0$ one has the identity $\chi(\mathcal{M}_{0,n})=(-1)^{n+1}(n-3)!$.
They also computed the ordinary Euler characteristics of $\mathcal{M}_{g,n}$ for $n=0,1$.
Bini and Harer in \cite{bini_harer_euler} obtain formulas for the Euler characteristics of 
$M_{g,n}$ and $\overline{M}_{g,n}$.
In \cite{gorsky_equivariant} Gorsky gives a formula for the $S_n$-equivariant Euler characteristics of 
the moduli spaces $M_{g,n}$.
In \cite{manin_trees} Manin studies generating functions in algebraic geometry and their relation with sums over trees. 
According to the result of Kontsevich \cite{kontsevich_enumeration} the computation of these generating functions is reduced to the
problem of finding the critical value of an appropriate formal potential.
Among other things, in \cite{manin_trees} Manin calculates the Betti numbers and Euler characteristics of the moduli spaces
$\overline{M}_{0,n}$ by showing that the corresponding generating functions satisfy certain differential equations
and are solutions to certain functional equations. 
In \cite{fulton_macpherson_compactification} Fulton and MacPherson give the computation of the virtual Poincar\'e polynomial of 
$\overline{M}_{0,n}$.
The equivariant version of their calculation is due to Getzler \cite{getzler_zero}, 
which also gives the formula for the open part $M_{0,n}$.
In \cite{getzler_semi, getzler_resolving} Getzler calculates the Euler characteristics of $M_{1,n}$ and $\overline{M}_{1,n}$. 
He gives the equivariant answer in the Grothendieck group of mixed Hodge structures where the natural action of the symmetric group $\Sigma_n$ is considered.

\subsection{The stable rational cohomology of $M_{g,n}$ and Mumford's conjecture}
According to Theorem \ref{stable}, there is an isomorphism between $H^k(M_g, \mathbb{Q})$
and $H^k(M_{g+1},\mathbb{Q})$ for $g \gg k$. This leads to the notion of the stable cohomology of $M_g$.
In low degree this happens very quickly:

\begin{thm}
The following are true:
\begin{enumerate}
\item[(i)] $H^1(M_{g,n}, \mathbb{Q})=0$ for any $g \geq 1$ and any $n$ such that $2g-2+n >0$.
\item[(ii)] $H^2(M_{g,n}, \mathbb{Q})$ is freely generated by $\kappa_1,\psi_1, \dots , \psi_n$ for any 
$g \geq 3$ and any $n$. $H^2(M_{2,n}, \mathbb{Q})$ is freely generated by $\psi_1, \dots , \psi_n$ for any $n$,
while $H^2(M_{1,n}, \mathbb{Q})$ vanishes for all $n$. 
\end{enumerate}
\end{thm}
The result on $H^1$ is due to Mumford for $n=0$ \cite{mumford_abelian} and to Harer \cite{harer_stability} for arbitrary $n$, 
the one on $H^2$ is due to Harer \cite{harer_second}.
The general form is Harer's stability theorem, which was first proven by Harer in \cite{harer_stability} 
and then improved by Ivanov \cite{ivanov_stabilization, ivanov_homology} and Boldsen \cite{boldsen_improved}.
\begin{thm} \label{stable}
(Stability Theorem). For $k< \frac{g-1}{2}$, there are isomorphisms
$$H^k(M_g, \mathbb{Q}) \cong H^k(M_{g+1}, \mathbb{Q}) \cong H^k(M_{g+2}, \mathbb{Q}) \cong \dots .$$
\end{thm}
According to this theorem, it makes sense to say that $k$ is in the stable range if $k<(g-1)/2$. 
In \cite{harer_improved} Harer extends the range to $k<\frac{2}{3}(g-1)$. 
The following result is due to Miller \cite{miller_homology} and Morita \cite{morita_characteristic}:

\begin{thm}
Let $k$ be in the stable range. Then  the natural homomorphism 
$$\mathbb{Q}[x_1, x_2, \dots] \rightarrow H^*(M_g,\mathbb{Q})$$
sending $x_i$ to $\kappa_i$ for $i=1,2, \dots$ is injective up to degree $2k$.
\end{thm}

Madsen and Weiss in \cite{madsen_weiss_stable} prove the Mumford's conjecture. More precisely, they show that:
\begin{thm} 
(Mumford's Conjecture). Let $k$ be in the stable range. Then the natural homomorphism 
$$\mathbb{Q}[x_1, x_2, \dots] \rightarrow H^*(M_g,\mathbb{Q})$$
sending $x_i$ to $\kappa_i$ for $i=1,2, \dots$ is an isomorphism up to degree $2k$.
\end{thm}
There are other proofs by Segal \cite{segal_definition}, Tillmann \cite{tillmann_homotopy}, Madsen and Tillmann \cite{madsen_tillmann_stable}.
The case of pointed curves is studied by Looijenga \cite{looijenga_stable}:

\begin{thm}
Let $k$ be in the stable range. Then the natural homomorphism
$$H^*(M_g,\mathbb{Q})[y_1, \dots, y_n] \rightarrow H^*(M_{g,n},\mathbb{Q})$$
sending $y_i$ to $\psi_i$ for $i=1, \dots , n$ is an isomorphism up to degree $2k$.
\end{thm}

Our main reference for this part is Chapter 5 of \cite{geometry_of_algebraic_curves_II}. For more details about stable cohomology 
and Mumford's conjecture see \cite{ekedahl_mumford,tillmann_stable,eliashber_galatius_mishachev_mw}.

\subsection{Non-tautological classes} There are non-tautological classes on the moduli spaces of curves. 
This phenomenon starts from genus one. For instance, there are many algebraic zero cycles on 
$\overline{M}_{1,11}$ 
while the tautological group 
$R_0(\overline{M}_{1,11})$ is a one dimensional vector space over $\mathbb{Q}$. 
The infinite dimensionality of the Chow group $A_0(\overline{M}_{1,11})$ follows from the existence of a holomorphic 11-form on the space. 
For more details see 
\cite{faber_pandharipande_taut_non_taut, graber_pandharipande_constructions,graber_vakil_tautological}. 
Notice that this argument does not produce an explicit algebraic cycle which is not tautological.  
Several examples of non-tautological classes on the moduli spaces of curves are constructed in \cite{graber_pandharipande_constructions}. 
To explain their result let us recall the relevant notation. The image of the tautological ring 
$$R^*(\overline{M}_{g,n}) \subset A^*(\overline{M}_{g,n})$$ 
under the cycle class map 
$$A^*(\overline{M}_{g,n}) \rightarrow H^*(\overline{M}_{g,n})$$
is denoted by $RH^*(\overline{M}_{g,n})$. 
In genus zero, the equality 
$$A^*(\overline{M}_{0,n})=R^*(\overline{M}_{0,n})=RH^*(\overline{M}_{0,n})=H^*(\overline{M}_{0,n})$$
holds due to Keel \cite{keel_intersection}. 
In genus one, in \cite{getzler_intersection} Getzler predicted the isomorphisms 
$$R^*(\overline{M}_{1,n}) \cong RH^*(\overline{M}_{1,n}),$$
$$RH^*(\overline{M}_{1,n}) \cong H^{2*}(\overline{M}_{1,n}),$$
for $n \geq 1$. The proof of these statements are given in \cite{petersen_1}
The following proposition is important in finding non-tautological classes:

\begin{prop}\label{rh}
Let $\iota: \overline{M}_{g_1,n_1 \cup \{*\}} \times \overline{M}_{g_2,n_2 \cup \{  \bullet \}} \rightarrow \overline{M}_{g,n_1+n_2}$ 
be the gluing map to a boundary divisor. 
If $\gamma \in RH^*(\overline{M}_{g_1+g_2,n_1+n_2})$, then $\iota^*(\gamma)$ has a tautological
K\"unneth decomposition: $$\iota^*(\gamma) \in RH^*(\overline{M}_{g_1, n_1 \cup \{ * \}} ) \otimes RH^*(\overline{M}_{g_1, n_1 \cup \{ \bullet \}} ).$$
\end{prop}

The main idea for detecting non-tautological classes is using Proposition \ref{rh} and the existence of odd cohomology on the moduli spaces of curves.
Its existence implies that the K\"unneth decomposition of certain pull-backs is not tautological. 
The first example of a non-tautological algebraic cycle is constructed on $\overline{M}_{2h}$, 
for all sufficiently large odd $h$. Here is the description of the cycle: Let $Y \subset \overline{M}_{2h}$
denote the closure of the set of nonsingular curves of genus $g:=2h$ which admit a degree 2 map to a
nonsingular curve of genus $h$. The pullback of $[Y]$ under the gluing morphism 
$$\iota: \overline{M}_{h,1} \times \overline{M}_{h,1} \rightarrow \overline{M}_g$$
is shown to be a positive multiple of the diagonal 
$\Delta \subset \overline{M}_{h,1} \times \overline{M}_{h,1}$.
According to Pikaart \cite{pikaart_orbifold}, for sufficiently large $h$, the space $\overline{M}_{h,1}$
has odd cohomology. This shows that $Y$ is not a tautological class and in fact its image
in $H^*(\overline{M}_g)$ does not belong to the distinguished part $RH^*(\overline{M}_g)$.
Another example is provided by looking at the locus $Z$ inside $\overline{M}_{2,20}$ 
whose generic point corresponds to a 20-pointed, nonsingular, bielliptic curve of genus two
with 10 pairs of conjugate markings. The intersection of $Z$ with the boundary map 
$$\iota: \overline{M}_{1,11} \times \overline{M}_{1,11} \rightarrow \overline{M}_{2,20}$$
gives a positive multiple of the diagonal whose cohomology class is not tautological.
By using the results of Getzler \cite{getzler_semi} regarding the cohomology of $\overline{M}_{1,n}$
it is proven that the class $[Z]$ is not tautological even on the interior $M_{2,20}$.
Finally, it is proven that 
\begin{thm}\label{22}
The push-forward $\iota_*[ \Delta ]$ of the diagonal in 
$A^*(\overline{M}_{1,12} \times \overline{M}_{1,12})$ via the boundary inclusion,
$$\iota: \overline{M}_{1,12} \times \overline{M}_{1,12} \rightarrow \overline{M}_{2,22}$$
is not a tautological class even in cohomology. 
\end{thm}
The proof is based on the existence of a 
canonical non-zero holomorphic form $s \in H^{11,0}(\overline{M}_{1,11},\mathbb{C})$ coming from the 
discriminant cusp form of weight 12. 
This gives odd cohomology classes on $\overline{M}_{1,12}$ and yields the desired result.  

\begin{rem}
In a similar way, one may wonder whether the push-forward of the diagonal
$$\Delta_{11} \subset \overline{M}_{1,11} \times \overline{M}_{1,11},$$
which has K\"unneth components not in $RH^*(\overline{M}_{1,11})$, via the morphism
$$\iota: \overline{M}_{1,11} \times \overline{M}_{1,11} \times \overline{M}_{2,20}$$
belongs to the tautological part of the cohomology of $\overline{M}_{2,20}$. 
As it is mentioned in \cite{faber_pandharipande_taut_non_taut} this question is unsolved.
\end{rem}
In \cite{faber_pandharipande_taut_non_taut} three approaches for detecting and studying non-tautological classes are presented: 
\subsubsection{Point counting and modular forms} 
The method of counting points over finite fields gives important information about the cohomology groups. For a finite field $\mathbb{F}_q$ the set of $\mathbb{F}_q$-points of $\overline{M}_{g,n}$ is denoted by 
$\overline{M}_{g,n}(\mathbb{F}_q)$. 
The idea of detecting non-tautological classes is based on the following fact:
If all cohomology classes of $M_{g,n}$ come from the fundamental classes of its subvarieties then the number
$|M_{g,n}(\mathbb{F}_q)|$ will be a polynomial of degree $d$ in $q$, 
where $d=3g-3+n$ is the dimension of $M_{g,n}$. Otherwise, $M_{g,n}$ has a non-algebraic cohomology class.
In particular, this cohomology class is not tautological. 
The first example of detecting non-tautological classes by the method of point counting occurs in genus one.
It is proven that the number of points of $M_{1,11}$ over finite fields can not be described by a polynomial of degree 11. This means that $M_{1,11}$ has a non-tautological class. This is explained by studying the cohomology of local systems in genus one which is expressed in terms of elliptic cusp forms. 

\subsubsection{Representation theory} 
There is a natural action of the symmetric group $\Sigma_n$ on $\overline{M}_{g,n}$ by permuting the markings. 
Therefore, the cohomology ring $H^*(\overline{M}_{g,n})$ becomes a $\Sigma_n$-module. 
The analysis of the action of the symmetric group on $H^*(\overline{M}_{g,n})$ provides a second approach to find non-tautological cohomology. 

The length of an irreducible representation of $\Sigma_n$ is defined to be the number of parts in the corresponding partition of $n$. For a finite dimensional representation $V$ of the symmetric group $\Sigma_n$ the length $\ell(V)$ is defined to be the maximum of the lengths of the irreducible constituents.

\begin{thm}
For the tautological rings of the moduli spaces of curves, we have
\begin{enumerate}
\item[(i)] $\ell(R^k(\overline{M}_{g,n})) \leq \min (k+1,3g-2+n-k,[\frac{2g-1+n}{2}]),$
\item[(ii)] $\ell(R^k(M_{g,n}^{ct})) \leq \min (k+1,2g-2+n-k),$
\item[(iii)] $\ell(R^k(M_{g,n}^{rt})) \leq \min (k+1,g-1+n-k).$
\end{enumerate}
\end{thm}  
Assuming the conjectural formula in \cite{faber_geer_genus2_1,faber_geer_genus2_2} for $H^*(\overline{M}_{2,n})$, 
several cohomology classes are obtained 
which can not be tautological because the length of the corresponding $\Sigma_n$-representation is too large.

\subsubsection{Boundary geometry}
We have seen that the existence of non-tautological cohomology for $\overline{M}_{1,11}$ 
leads to the construction of several non-tautological algebraic cycles on the moduli spaces of curves. 
In \cite{faber_pandharipande_taut_non_taut} new non-tautological classes in $H^*(\overline{M}_{2,21}, \mathbb{Q})$ are found. 
They define the \emph{left diagonal} cycle $$\Delta_L \subset \overline{M}_{1,12} \times \overline{M}_{1,11}$$
and the \emph{right diagonal} 
$$\Delta_R \subset \overline{M}_{1,11} \times \overline{M}_{1,12}.$$
It is shown in \cite{faber_pandharipande_taut_non_taut} that the push-forwards $\iota_{L*}[\Delta_L]$ and $\iota_{R*}[\Delta_R]$ under
the boundary morphisms 
$$\iota_L: \overline{M}_{1,12} \times \overline{M}_{1,11} \rightarrow \overline{M}_{2,21},$$
$$\iota_R: \overline{M}_{1,11} \times \overline{M}_{1,12} \rightarrow \overline{M}_{2,21},$$
are not tautological. 
This gives another proof of Theorem \ref{22} since the push-forward of the diagonal class $\iota_*[\Delta]$ via the canonical morphism $$\pi:\overline{M}_{2,22} \rightarrow \overline{M}_{2,21}$$
is the same as the class $\iota_{L*}[\Delta_L]$, which was seen to be non-tautological. 
The result follows as the push-forwards of tautological classes on $\overline{M}_{2,22}$
via the morphism $\pi$ belong to the tautological ring of $\overline{M}_{2,21}$.

\subsection{Cohomology of local systems}

Let $A_g$ be the moduli space of principally polarized abelian varieties of dimension $g$.
Denote by $\V$ the local system on $A_g$ whose fiber at the moduli point $[X] \in A_g$ is the symplectic vector space $H^1(X)$.
For a given partition $\la$ with at most $g$ non-zero parts there is a local system $\V_{\la}$ on $A_g$.
It is the irreducible representation of highest weight in 

$$Sym^{\la_1-\la_2} (\wedge^1 \V) \otimes Sym^{\la_2-\la_3} (\wedge^2 \V) \otimes \dots \otimes Sym^{\la_{g-1}-\la_g}( \wedge^{g-1} \V) \otimes (\wedge^g \V)^{\la_g}.$$

The analysis of these local systems and computing their cohomology is an interesting question.
In genus one the results are classical and the cohomology of local systems are expressed in terms of elliptic modular forms via the Eichler-Shimura isomorphism.
In genus two the story becomes more complicated.
According to a conjecture of Faber and van der Geer \cite{faber_geer_genus2_1} there are explicit formulas for the cohomology of local systems in genus two in terms of Siegel and elliptic modular forms. 
Their conjecture is proved for the regular highest weight in \cite{weissauer_trace, tehrani_strict} and for the general case in \cite{petersen_abelian_surfaces}. 
In genus three there are analogue conjectures by Bergstr\"om, Faber and van der Geer \cite{bergstrom_faber_geer_three}.
The knowledge of cohomology of such local systems is quite useful in the study of cohomology of moduli of curves.
In \cite{petersen_1, petersen_2} they were used to determine the cohomology of moduli of curves in genus two.
%To state the results we consider the natural projection of $M_{1,n}$ onto $M_{1,1}$ which only remembers the first marking.
%The local system $\V_k$, whose fiber over the moduli point $(C,p)$ is Sym$^k H^1(C)$ is defined on $M_{1,1}$.
%The Eichre-Shimura isomorphism identifies the cohomology group $H^1_c()$
%This isomorphism can be interpreted on moduli of curves as follows:
%$$H^1(M_{1,1},\V_k) \otimes \C \cong H^{k+1,0} \oplus H^{0,k+1} \oplus H^{k+1,k+1},$$
%and each summand is identified with a certain space of modular forms. 
The analysis of local systems for $A_g$ and $M_g$ for higher $g$ becomes much harder and at the moment there are not even conjectural description of their cohomology.
%According to the Langlands philosophy or perhaps beyond that picture one should expect modular forms.
While we are still far from understanding the cohomology of local systems in general they play a crucial role in the study of tautological classes.
In a work on progress with Petersen and Yin \cite{POWR} we study such connection.
The \emph{tautological cohomology} of local systems is studied.
It turns out that very explicit computations are possible to extract useful information about the tautological ring.
Using computer we have been able to find the tautological cohomology of all local systems in genus 3 and 4.
This gives a complete description of tautological rings of $C_g^n$ and $M_{g,n}^{rt}$ in genus 3 and 4 for every $n$ based on the results proved in \cite{tavakol_conjectural, petersen_poincare}.
There is also a conjectural description of the tautological cohomology of certain local systems in genus 5.
These conjectures give an interesting link to the Gorenstein conjectures.
One can show that the non-vanishing of the forth cohomology of the local system $\V_{2^4}$ on $M_5$ would give a new counterexample to the Gorenstein conjectures.
For a large class of local systems we have an explicit method which conjecturally determines their tautological cohomology.
Unfortunately, the local system $\V_{2^4}$ does not belong to this class.
For more details we refer to the recent report \cite{POWR} at the workshop Moduli spaces and Modular forms at Oberwolfach.

\part*{Integral computations on the moduli stack}
In \cite{mumford_picard} Mumford studies topologies defined for a moduli problem and the right way of defining its invariants. 
An invertible sheaf $L$ on the moduli problem $\mathcal{M}_g$ consists in two sets of data:

\begin{itemize}
\item
For all families $\pi: \mathcal{X} \rightarrow S$ of nonsingular curves of genus $g$, an invertible sheaf $L(\pi)$ on $S$.

\item 
For all morphisms $F$ between such families:
$$\begin{CD}
 \mathcal{X}_1  @>>>  \mathcal{X}_2 \\
 @VV \pi_1 V @VV \pi_2 V \\
 S_1 @> f >> S_2  
 \end{CD}$$

an isomorphism $L(F)$ of $L(\pi_1)$ and $f^*(L(\pi_2))$ satisfying a natural cocycle condition.
\end{itemize}
 
 He gives two different ways to describe the Picard group associated to the moduli problem in genus one. 
In particular, he determines the Picard group of $\mathcal{M}_{1,1}$ over the spectrum of a field of characteristic not 2 or 3. 
This is a cyclic group of order 12 generated by the class of the Hodge bundle.
Intersection theory on algebraic stacks and their moduli spaces is developed in \cite{vistoli_intersection} by Vistoli.
The special case of quotients by actions of reductive algebraic groups is treated in \cite{vistoli_chow}.
Fulton and Olsson in \cite{fulton_olsson_picard} consider a general base scheme $S$ and compute the Picard group 
$Pic(\mathcal{M}_{1,1,S})$ of the fiber product 
$\mathcal{M}_{1,1,S}:=S \times_{\text{Spec}(\mathbb{Z})} \mathcal{M}_{1,1}$ 
as well as the Picard group $Pic(\overline{\mathcal{M}}_{1,1,S})$ 
for the fiber of the Deligne-Mumford compactification 
$\overline{\mathcal{M}}_{1,1}$ of $\mathcal{M}_{1,1}$ over $S$. 
To state their result we recall the definition of the Hodge class $\lambda$ on the moduli stack: 
Every morphism $t:T \rightarrow  \mathcal{M}_{1,1}$ corresponds to an elliptic curve $f:E \rightarrow T$. 
The pullback $t^* \lambda$ is the line bundle $f_* \Omega^1_{E/T}$. 
The canonical extension of this bundle to the compactified space $\overline{\mathcal{M}}_{1,1}$
is denoted by the same letter $\lambda$.
They prove the following results:

\begin{thm}
Let $S$ be a scheme. Then the map $$\mathbb{Z}/(12) \times Pic(\mathbb{A}^1_S) \rightarrow  Pic(\mathcal{M}_{1,1,S}), \qquad (i, \mathcal{L}) \mapsto \lambda^{\otimes i} \otimes p^* \mathcal{L}$$
is an isomorphism if either of the following hold:

\begin{enumerate}
\item[(i)] $S$ is a $\mathbb{Z}[1/2]$-scheme.
\item[(ii)] $S$ is reduced.
\end{enumerate}
\end{thm}

\begin{thm}
The map $$\mathbb{Z} \times Pic(S) \rightarrow Pic(\overline{\mathcal{M}}_{1,1,S}), \qquad (n,M) \mapsto \lambda^n \otimes_{\mathcal{O}_S} M $$
is an isomorphism for every scheme $S$.
\end{thm}

In \cite{mumford_stability} Mumford proves that the Picard group of the moduli functor $\mathcal{M}_{g,n}$ has no torsion and contains the Picard group $\text{Pic}(M_{g,n})$ of the moduli space as a subgroup of finite index. 
In \cite{arbarello_cornalba_picard} Arbarello and Cornalba study the Picard group of the moduli spaces of curves.  
They give an explicit basis for the Picard group of $\mathcal{M}_{g,n}$ and its compactification
$\overline{\mathcal{M}}_{g,n}$ by means of stable curves. They first prove the following result for $n=0$:

\begin{thm}
For any $g \geq 3$, $Pic(\overline{\mathcal{M}}_g)$ is freely generated by $\lambda,\delta_0, \dots , \delta_{[g/2]}$,
while $Pic(\mathcal{M}_g)$ is freely generated by $\lambda$.
\end{thm}

Their proof goes as follows: From Harer's theorem it follows that any class in the Picard group of 
$\overline{\mathcal{M}}_g$ 
is a linear combination of $\lambda$ and the $\delta_i$'s with rational coefficients. By finding suitable families of curves and taking degrees it is shown that these coefficients are indeed integers and the generators are linearly independent. 
The generalization of this result considers pointed curves: 
\begin{thm}
For every $g \geq 3$, $Pic(\overline{\mathcal{M}}_{g,n})$ is freely generated by $\lambda$, the $\psi$'s and the $\delta$'s,
while $Pic(\mathcal{M}_{g,n})$ is freely generated by $\lambda$ and the $\psi$'s.
\end{thm}

Another interesting question is to determine the Picard group of the moduli space $\overline{M}_{g,n}$.   
This is the content of the following proposition proven in \cite{arbarello_cornalba_picard}:
\begin{prop}
If $g \geq 3$, $Pic(\overline{M}_{g,n})$ is the index two subgroup of $Pic(\overline{\mathcal{M}}_{g,n})$ generated by 
$\psi_1, \dots , \psi_n, 2 \lambda, \lambda+\delta_1$, and the boundary classes different from $\delta_1$.
\end{prop}

\subsection{The Picard group of the moduli stack of hyperelliptic curves}

In \cite{arsie_vistoli_stacks} Arsie and Vistoli study the moduli of cyclic covers of projective spaces. 
As a result they prove that the Picard group of the moduli stack $\mathcal{H}_g$ of smooth hyperelliptic curves of genus $g \geq 2$ is finite cyclic. This group is shown to be isomorphic to $\mathbb{Z}/(4g+2) \mathbb{Z}$ for $g$ even, and to $\mathbb{Z}/2(4g+2) \mathbb{Z}$ for $g$ odd. 
In \cite{gorchinskiy_viviani_picard} Gorchinskiy and Viviani describe the stack $\mathcal{H}_g$ as a quotient stack. 
Their description leads to a geometric construction of generators for the Picard group.
In \cite{cornalba_harris_divisor} Cornalba and Harris show that in characteristic zero, the identity 
\begin{equation}\label{odd}\tag{3}
(8g+4) \lambda=g \xi_{irr}+2 \sum_{i=1}^{[g-1/2]} (i+1)(g-i) \xi_i +4 \sum_{j=1}^{[g/2]} j(g-j) \delta_j
\end{equation}
holds in the rational Picard group of the closure $\overline{\mathcal{H}}_g$ of $\mathcal{H}_g$ inside 
$\overline{\mathcal{M}}_g$. 
Here, the class $\lambda$ stands for the Hodge class and the boundary classes  
$\xi_{irr}, \xi_1, \dots , \xi_{[g-1/2]} , \delta_1, \dots , \delta_{[g/2]}$ denote the irreducible components 
of the complement of $\mathcal{H}_g$ in $\overline{\mathcal{H}}_g$.  
For the precise definition of these classes see \cite{cornalba_picard} or \cite{cornalba_harris_divisor}.
Cornalba \cite{cornalba_picard} has improved this result by showing that this identity is valid already 
in the integral Picard group of 
$\overline{\mathcal{H}}_g$. 
Cornalba observes that half of the relation \eqref{odd} holds in the Picard group of the moduli stack
$\overline{\mathcal{H}}_g$: 
\begin{equation}\label{even}\tag{4}
(4g+2) \lambda=g \xi_{irr}/2+ \sum_{i=1}^{[g-1/2]} (i+1)(g-i) \xi_i + 2\sum_{j=1}^{[g/2]} j(g-j) \delta_j
\end{equation}
These two related identities together give a complete description of the 
Picard group when $g$ is not divisible by 4.
More precisely, Cornalba shows that when $g \geq 2$ is an integer as above, then 
Pic$(\overline{\mathcal{H}}_g)$ is generated by $\lambda$ and by the boundary classes. 
The relations between these classes are generated by \eqref{odd} when $g$ is odd and by
\eqref{even} when $g$ is even. 
When $g$ is divisible by 4 the situation becomes more subtle. 
In fact, as observed by Gorchinskiy and Viviani in \cite{gorchinskiy_viviani_picard} the Picard group of the open part 
$\mathcal{H}_g$ is not generated by the Hodge class in this case.
Cornalba solves this problem by introducing a geometrically defined line bundle $\mathcal{Z}$
on $\overline{\mathcal{H}}_g$, whose class together with the boundary divisors 
gives a set of generators for the Picard group.
He also describes the space of relations among the generators. 

\subsection{Equivariant Intersection Theory} 
In \cite{vistoli_equivariant}, Vistoli studies equivariant Grothendieck groups and equivariant Chow groups. 
To mention his result let us recall the related notions for the classical case. 
Let $X$ be a smooth separated scheme of finite type over a field $k$ and denote by 
$K_0(X)$ the Grothendieck ring of vector bundles on $X$ tensored with the field $\mathbb{Q}$ of rational numbers. 
The Chern character $$ch:K_0(X) \rightarrow A^*(X) \otimes \mathbb{Q},$$
defines a ring homomorphism, which commutes with pullback. 
When the scheme $X$ is not necessarily smooth one can consider $K_0^{'}(X)$ to be the Grothendieck group of coherent sheaves on $X$ tensored with $\mathbb{Q}$. 
In this situation there is a group isomorphism $$\tau_X: K_0^{'}(X) \rightarrow A_*(X) \otimes \mathbb{Q},$$
called the Riemann-Roch map, which commutes with proper pushforward. 
To study the equivariant case Vistoli considers an algebraic group $G$ over $k$ acting properly on a separated scheme $X$ of finite type over $k$ in such a way that the stabilizer of any geometric point of $X$ is finite and reduced. Vistoli defines $K_0^{'}(X//G)$ as the Grothendieck group of $G$-equivariant coherent sheaves on $X$ tensored with $\mathbb{Q}$. 
Under the assumption of the existence of a geometric quotient $X/G$ he proves the existence of an equivariant Riemann-Roch 
map $$\tau_X: K_0^{'}(X//G) \rightarrow A^*(X/G) \otimes \mathbb{Q},$$
which is surjective, but in general not injective. He states a conjecture about its kernel and proves it when $X$ is smooth.
In \cite{gillet_intersection}, Gillet proves a Riemann-Roch theorem for algebraic spaces. It is natural to ask whether there is a Riemann-Roch theorem for algebraic stacks. The Riemann-Roch theorem for algebraic stacks states that there is a homomorphism from the $K$-theory of coherent sheaves on a stack to its Chow groups, which commutes with proper pushforwards along representable morphisms of stacks. 
These approaches are based on the analysis of the invariant cycles on the variety under the group action. 
But in general there are not enough invariant cycles on $X$ to define equivariant Chow groups with nice properties, such as having an intersection product when $X$ is smooth. A development of equivariant intersection theory for actions of linear algebraic groups on schemes and algebraic spaces is presented in 
\cite{edidin_graham_equivariant} by Edidin and Graham. 
They give a construction of equivariant Chow groups which have all the functorial properties of ordinary Chow groups. 
In this approach, instead of considering only invariant cycles on $X$ they define an equivariant class to be represented by an invariant cycle on $X \times V$, 
where $V$ is a representation of $G$. 
The motivation for this definition is the combination of these two facts:
For a vector bundle $E$ on an algebraic variety $X$ we have that $A_i(E)=A_i(X)$,
and the vector bundle $X \times V$ on $X$ descends to a vector bundle on the quotient of $X$ by $G$. 
For an action of the linear algebraic group $G$ on a space $X$ the $i^{th}$ equivariant Chow group of 
$X$ is denoted by $A_i^G(X)$. 

\subsubsection{Functorial properties}
Let $\bold{P}$ be one of the properties of morphisms of schemes or algebraic spaces such as proper, flat, regular embedding or local complete intersection. The following fact is proven in \cite{edidin_graham_equivariant}:

\begin{prop}
Equivariant Chow groups have the same functoriality as ordinary Chow groups for equivariant morphisms with property $\bold{P}$.
\end{prop}
 
Theorem \ref{eq} in \cite{edidin_graham_equivariant} says that the rational Chow groups of the quotient of a variety by a group acting with finite stabilizers can be identified with the equivariant Chow groups of the original variety. If the original variety is smooth then the rational Chow groups of the quotient inherit a canonical ring structure (Theorem \ref{ring}). 
The theory of equivariant intersection theory provides a powerful tool in computing the rational Chow groups of a moduli space which is a quotient of an algebraic space $X$. This has a natural ring structure when $X$ is smooth.
Let us first recall the relevant definitions and notations. 

\subsubsection{Equivariant higher Chow groups}
Bloch \cite{bloch_algebraic} defines higher Chow groups $A^i(X,m)$ for a quasi-projective scheme over a field $k$ and an integer $m$ as the $m^{th}$ homology group $H_m(Z^i(X,.))$ of a complex whose $p^{th}$ term is the group of cycles of codimension $i$ in the product $X \times \Delta^p$ of $X$ with which the faces intersect properly. 
Here, $\Delta_p$ denotes the standard cosmiplicial scheme over $k$ defined by 
$$\Delta^p=\text{Spec} \ k[t_0 , \dots , t_p]/(\sum_{i=0}^p t_i-1).$$
In a similar way the group $A_p(X,m)$ is defined. 

\subsection{Intersection theory on quotients}
Let $G$ be a $g$-dimensional group acting on an algebraic space $X$. Vistoli defines a quotient $\pi:X \rightarrow Y$
to be a map which has the following properties:

\begin{itemize}
\item
$\pi$ commutes with the action of $G$, the geometric fibers of $\pi$ 
are the orbits of the geometric points of $X$,

\item
$\pi$ is universally submersive, i.e., $U \subset Y$ is open if and only if $\pi^{-1}(U)$ is, and this property is preserved by base change.
\end{itemize}

\begin{thm}\label{eq}
\begin{enumerate}
\item[(a)] Let $X$ be an algebraic space with a locally proper $G$-action and let $\pi: X \rightarrow Y$ be a quotient. Then
$$A_{i+g}^G(X) \otimes \mathbb{Q} \cong A_i(Y) \otimes \mathbb{Q}.$$

\item[(b)] If in addition $X$ is quasi-projective with a linearized $G$-action, and the quotient $Y$ is quasi-projective, then
$$A_{i+g}^G(X,m) \otimes \mathbb{Q} \cong A_i(Y,m) \otimes \mathbb{Q}$$
\end{enumerate}
 \end{thm}

\begin{thm}\label{ring}
With the same hypotheses as in Theorem \ref{eq}, there is an isomorphism of operational Chow rings
$$\pi^*:A^*(Y)_{\mathbb{Q}} = A^*_G(X)_{\mathbb{Q}}.$$
Moreover if $X$ is smooth, then the map $A^*(Y)_{\mathbb{Q}} \rightarrow A_*(Y)_{\mathbb{Q}}$ is an isomorphism. 
In particular, if $X$ is smooth, the rational Chow groups of the quotient space $Y=X/G$ have a ring structure, which is independent of the presentation of $Y$ as a quotient of $X$ by $G$.
\end{thm}

\subsection{Integral Chow groups of quotient stacks}
When the group $G$ acts on an algebraic space $X$, the quotient $[X/G]$ exists in the category of stacks. 
In \cite{edidin_graham_equivariant} the relation between equivariant Chow groups and Chow groups of quotient stacks is discussed. 
In particular, it is shown that for proper actions, equivariant Chow groups coincide with the Chow groups defined by Gillet in terms of integral sub-stacks. 
It is also proven that the intersection products of Gillet and Vistoli are the same.  
For a quotient stack $\mathcal{F}=[X/G]$ the integral Chow group $A_i(\mathcal{F})$ is defined to be $A_{i+g}^G(X)$, where 
$g=\dim G$. It is proven in \cite{edidin_graham_equivariant} that the equivariant Chow groups do not depend on the representation as a quotient and hence they are invariants of the stack. The following fact shows that there is a natural product of the sum of the integral Chow groups of a quotient stack:

\begin{prop}
If $\mathcal{F}$ is smooth, then $\oplus_i A_i(\mathcal{F})$ has an integral ring structure. 
\end{prop} 

It is also proven in \cite{edidin_graham_equivariant} that the first equivariant Chow group $A^1_G(X)$ coincides with the Picard group of the moduli problem defined by Mumford for the stack $\mathcal{F}=[X/G]$:

\begin{prop}
Let $X$ be a smooth variety with a $G$-action. Then $A^1_G(X)=Pic_{fun}([X/G])$.
\end{prop}

There is also the notion of the integral operational Chow ring associated to any stack $\mathcal{F}$: 
An element $c \in A^k(\mathcal{F})$ defines an operation $$c_f: A_*(B) \rightarrow A_{*-k} (B)$$
for any map of schemes $f:B \rightarrow \mathcal{F}$. This naturally has a ring structure. 
The following fact shows that equivariant Chow rings coincide with the operational Chow rings: 

\begin{prop}
Let $\mathcal{F}=[X/G]$ be a smooth quotient stack. Then $A^*(\mathcal{F})=A^*_G(X)$.
\end{prop}

\subsection{The Chow ring of the moduli stack of elliptic curves}
As an application of this theory the integral Chow ring of the moduli stack of elliptic curves is computed in \cite{edidin_graham_equivariant}.
Recall that a section of $\mathcal{M}_{1,1}$ over a scheme $S$ is the data $(\pi:C \rightarrow S,\sigma)$, 
where $\pi$ is a smooth curve of genus one and $\sigma:S \rightarrow C$ is a smooth section. 
The stack $\overline{\mathcal{M}}_{1,1}$ is defined in a similar way, with the difference that the fibers are $\pi$
are assumed to be stable nodal curves. 
It is easy to see that these stacks are quotients of smooth varieties by actions of linear algebraic groups.
The computation of the equivariant intersection ring of these quotient stacks gives the following result:
\begin{prop}
$$A^*(\mathcal{M}_{1,1})=\mathbb{Z}[t]/(12t), \qquad A^*(\overline{\mathcal{M}}_{1,1})=\mathbb{Z}[t]/(24t^2),$$
where $t$ is the first Chern class of the Hodge bundle.
\end{prop}

\subsection{The integral Chow ring of the moduli stack of smooth curves of genus two}
In the appendix of \cite{edidin_graham_equivariant} Vistoli considers the case of genus two and computes the Chow ring of $\mathcal{M}_2$. 
He proves that $\mathcal{M}_2$ is a quotient of an algebraic space by a linear algebraic group:
Let $Y$ be the stack whose objects are pairs $(\pi,\alpha)$, where $\pi:C \rightarrow S$ is a smooth proper
morphism of schemes whose fibers are curves of genus 2, and $\alpha$ is an isomorphism of $\mathcal{O}_S$
sheaves $\alpha: \mathcal{O}_S^{\oplus 2} \cong \pi_* \omega_{\pi}$, where $\omega_{\pi}$ is the relative dualizing
sheaf of $\pi$; morphisms are defined in the natural way as canonical pull-backs. 
The objects of $Y$ have no automorphisms, so that $Y$ is an algebraic space.
There is a natural left GL$_{2,k}$ action on $Y$: if $(\pi,\alpha)$ is an object of $Y$ with basis $S$ and $A \in$ GL$_2(S)$,
the action of $A$ on the pair $(\pi,\alpha)$ is defined by $$A \cdot (\pi,\alpha)=(\pi,\alpha \circ A^{-1}).$$
The moduli stack $\mathcal{M}_2$ is equal to the quotient $Y/\text{GL}_2$.
Vistoli identifies $Y$ with the open subset of all binary forms $\phi(x)=\phi(x_0,x_1)$ in two variables of degree 6, denoted by $X$,
consisting of non-zero forms with distinct roots. The given action of GL$_2$ corresponds to the action of
GL$_2$ on $X$ defined by $$A \cdot \phi(x)=\text{det} (A)^2 \phi(A^{-1}x).$$
The calculation of the equivariant Chow ring $A^*_{\text{GL}_2}(X)$ of $X$ leads to the following result:

\begin{thm}
Assume that $k$ has characteristic different from 2 and 3. Then $$A^*(\mathcal{M}_2)=\frac{\mathbb{Z}[\lambda_1,\lambda_2]}{(10\lambda_1, 2\lambda_1^2-24 \lambda_2)},$$
where $\lambda_i=c_i(\mathbb{E})$ is the $i^{th}$ Chern class of the Hodge bundle.
\end{thm}

The Chow rings of moduli spaces of maps of projective spaces and the Hilbert scheme of rational normal curves 
are computed by Pandharipande in 
\cite{pandharipande_chow, pandharipande_intersections} 
using equivariant methods.

%\appendix
%\section{Algebraic stacks and their intersection theory}
%\subsection{First appendix}
%\subsection{Second appendix}

%\section{Geometric invariant theory}

\bibliographystyle{amsplain}
\bibliography{mybibliography}
\end{document}